\documentclass[11pt, a4paper]{amsart}
\usepackage{graphicx} 
\usepackage{amsfonts, amssymb, amsmath, amsthm, amscd, mathtools, multicol, tikz, tikz-cd, caption, enumerate, mathrsfs, thmtools, cite, enumitem}
\usetikzlibrary{shapes, arrows.meta, positioning}
\usepackage{inputenc}
\usepackage[foot]{amsaddr}
\usepackage[pagebackref=true, colorlinks, linkcolor=blue, citecolor=red]{hyperref}
\usepackage{latexsym}
\usepackage{fullpage}
\usepackage{microtype}
\usepackage{subfiles} 
\setlist{topsep=0.5em, itemsep=1em} 

\makeatletter
\makeindex 
\newcommand{\be}{\begin{equation}}
\newcommand{\ee}{\end{equation}}
\newcommand{\beano}{\begin{eqn*}} 
\newcommand{\eeano}{\end{eqnarray*}}
\newcommand{\ba}{\begin{array}}
\newcommand{\ea}{\end{array}}
\declaretheoremstyle[headfont=\normalfont]{normalhead}
\theoremstyle{plain} 
\newtheorem{theorem}{Theorem}[section]
\theoremstyle{definition}

\theoremstyle{definition}

\newtheorem{lemma}[theorem]{Lemma}
\newtheorem{corollary}[theorem]{Corollary}
\newtheorem{proposition}[theorem]{Proposition}
\newtheorem{remark}[theorem]{Remark}
\newtheorem{example}[theorem]{Example}
\newtheorem{question}[theorem]{Question}

\newcommand{\Aut}{\mathrm{Aut}}

\newcommand{\Sl}{\mathfrak{sl}}
\newcommand{\F}{\mathbb{F}}
\numberwithin{equation}{section}

\newcommand{\img}{\mathrm{image}}

\begin{document}
\title{Images of Lie Polynomials on simple Lie algebras}
\author{Harish Kishnani}
\email{harishkishnani11@gmail.com}
\author{Anupam Singh}
\email{anupamk18@gmail.com}
\address{Indian Institute of Science Education and Research Pune, Dr Homi Bhabha Road, Pashan, Pune 411008 India}

\thanks{The first-named author acknowledges the support of ANRF-NPDF grant PDF/2025/002961 for the fellowship. The second-named author is funded by an ANRF-MATRICS Grant ANRF/ARGM/2025/000095/MTR}

\begin{abstract}
A Lie polynomial is an element of a free Lie algebra $\mathcal F_k$ on $k$-generators, which defines a Lie map on a given Lie algebra $L$, by substituting $k$-elements of $L$. Similar to word maps on groups and polynomial maps on algebras, one studies here questions analogous to Waring-like problems, the L'vov-Kaplansky conjecture, etc. In this article, we would like to address a problem for Lie algebras parallel to the one Lubotzky solved (Images of word maps in finite simple groups, Glasg. Math. J., 56, no. 2, 465-469, 2014) for finite simple groups. It is easy to verify that the image of a Lie map is (a) closed under automorphism, and (b) contains $0$. In this article, we prove that for a simple Chevalley algebra over a finite field of ``very good'' characteristic, these two properties are enough to classify all possible subsets that can be the image of a Lie polynomial. The next question is to find such Lie polynomials for a given subset satisfying the two properties. Contrary to the results over an algebraically closed field, we find Lie polynomials in the case of Lie algebra $\Sl_2(q)$, for $q$ odd, which give each $\rm{GL}_2(q)$ conjugacy class together with zero as an image.  
\end{abstract}
\maketitle
\section{Introduction}
An element $w$ of a free group on $k$-generators defines a word map on a group $G$. Images and fibers of such maps (see, for example~\cite{LST11, GKP16, GLOST18}) have been extensively studied in the context of Waring-like problems, including Ore's conjecture. Several fundamental results have been proved in this regard over the last $3$ decades, especially for finite simple groups, Lie groups, and algebraic groups. An analogous problem, motivated by the Waring problem, has been studied over both associative and non-associative algebras (see, for example~\cite{KMR16,  GD22, BS23, GD23, KS25, PSS25, PSS26}). A polynomial $f(x_1,\ldots, x_k)\in F<x_1,\ldots, x_k>$ in $k$-variables (non-commuting) defines a polynomial map on a given associative $F$-algebra $\mathcal A$ by evaluation. Surjectivity of such maps and when it is not, whether a finite sum of images (analogous to the Waring problem) gives the whole algebra, are some of the main questions in the subject. One of the most notable conjectures here is the L'vov-Kaplansky conjecture, which asks for the image of a multilinear polynomial on a matrix algebra over an infinite field. In this article, we focus on similar questions over Lie algebras.  

Let $L$ be a Lie algebra over a field $F$ and $\mathcal{F}_k$ be a free Lie algebra of rank $k$ over $F$ generated by $x_1, x_2, \ldots, x_k$. An element $w \in \mathcal{F}_k$ is called a Lie polynomial. A Lie polynomial $w \in \mathcal{F}_k$, induces a {\bf Lie polynomial map} $\tilde{w} \colon L^k \rightarrow L$, defined by sending a $k$-tuple $(v_1, v_2, \ldots, v_k)$ to $w(v_1, v_2, \ldots, v_k)$. As customary in the subject, we use the same notation $w$ for the Lie polynomial and the induced Lie polynomial map instead of $\tilde{w}$. The image of the Lie polynomial map induced by a word $w$ is denoted by $w(L)$. A simple example would be $w=[\ ,\ ]$, the Lie bracket on $L$, and the question would be if this induces a surjective map. This has been studied by several authors (see, for example,~\cite{Br63, Ak15, Kr23}) and doesn't seem to be fully solved over an arbitrary field. Questions about more general Lie polynomials have been studied in the literature too; see, for example~\cite{BGKP12, AEV15, KMR17, NR23, CF24}, but seem quite challenging. We are interested in studying what kind of subset of a Lie algebra $L$ could be a possible image of a Lie polynomial map, which we will make more precise below. 

It is quite easy to see that the image of a word map on a group is closed under automorphisms and contains the identity. In a brilliant work~\cite{Lu14}, Lubotzky classified that any subset $A$ of a non-Abelian finite simple group $G$ which is closed under automorphism, i.e., $\phi(A)\subset A, \forall \phi\in \Aut(G)$, and contains the identity is an image of a word in two variables. In~\cite{Ch90}, Chuang classified images of a polynomial map without constants on matrix algebra over $\mathbb F_q$ and showed that these are exactly the ones that contain $0$ and are closed under conjugation by invertible matrices (in view of Skolem-Noether theorem, these are all automorphisms). We wish to consider the analogous question here for finite simple Lie algebras. 

Let $L$ be a simple Lie algebra over $\mathbb F_q$ of Chevalley type (also called Chevalley Lie algebra or Chevalley algebra). We refer~\cite{Ca89, Se67} for the general theory. We assume the characteristic of the field $\mathbb F_q$ is a ``very good" prime. Note that the Lie algebra may have a center (e.g. $\Sl_n(q)$ for certain cases) over $\mathbb F_q$, hence will be simple modulo its center. Let $\Aut(L)$ denote the set of Lie algebra automorphisms for the Lie algebra $L$. Then, a subset $A\subset L$ has the following properties:
\begin{enumerate}
\item $0 \in A$.
\item $\phi(A) = A$, for all $\phi \in \Aut(L)$.
\end{enumerate}
We ask the following questions:
\begin{question}\label{question1}
Let $A\subset L$ satisfying the two above properties. Does there exist a Lie polynomial $w$ in finitely many variables such that $w(L)=A$? 
\end{question}

\begin{question}\label{question2}
If the answer to the Question~\ref{question1} is positive, can we find some such Lie polynomials that give the set $A$ as its image?
\end{question}

In this article, we answer the first question affirmatively for Chevalley Lie algebras over finite fields of ``very good'' characteristic. 
Our main theorem is following:  
\begin{theorem}
Let $L$ be a simple Chevalley algebra over a finite field $F$ of ``very good'' characteristic. Let $k(L)$ denote the minimum number such that for each non-zero element $a \in L$, there exists elements $a_2, a_3, \dots, a_{k(L)}$ in $L$ such that the set $\{a_1=a, a_2, a_3, \dots, a_{k(L)}\}$ generates $L$. Then, every subset $A\subset L$ containing $0$ and closed under automorphisms of $L$ is an image of some Lie polynomial from $\mathcal F_{k(L)}$.
\end{theorem}
\noindent The number $k(L)$ is a finite number because of dimension reasons, although an absolute constant is desirable. For finite simple groups, it is $2$ what is called one-and-a-half generation, but it doesn't seem to be known for all finite simple Lie algebras (over an algebraically closed field, it is known to be $2$). For a brief discussion on this, see Subsection~\ref{subsectionkL}. The main Theorem is proved in Section~\ref{Lubotzky'stheorem}, Theorem~\ref{Thm: Image canditates are images}. We collect several basic examples related to the Theorem in Section~\ref{examples}.

The second question is more challenging. In~\cite{BGKP12}, Bandman et. al. showed dominance of Lie maps on Chevalley Lie algebras. In~\cite{KMR17}, Kanel-belov et. al. showed that over an algebraically closed field of characteristic $\neq 2$, the possible images of Lie polynomials on $\Sl_2$ are the whole set $\Sl_2$, $\{0\}$, or the subset consisting of all non-nilpotent elements. Contrary to this, in Section~\ref{image1}, \ref{image2}, \ref{image3}, \ref{image4}, we take the Lie algebra $\Sl_2(q)$ and find Lie polynomials, often a combination of Engel polynomials, and show that each of the automorphism orbits together with $0$ can be obtained as an image. Thus, a single conjugacy class of a semisimple element together with zero is an image of some specific Lie polynomial. This is summarised in Theorem~\ref{Thm: Semisimple orbits for each integer}. We hope to continue this work in the future to solve the Problem~\ref{question2}.

\subsection{Acknowledgment}
The second-named author would like to thank Prof. Boris Kunyavskii and Prof. Pralay Chatterjee for several helpful discussions during their visit to the Lodha Mathematical Science Institute, Mumbai.

\section{Preliminaries}

Let $\mathfrak g$ be a finite-dimensional simple Lie algebra over $\mathbb {C} $. We know that for such an algebra, there exists a Chevalley basis that defines the Lie algebra over $\mathbb Z$, denoted as $\mathfrak g_{\mathbb Z}$. Let $F$ be a field. One defines a Lie algebra $\mathfrak g_F = \mathfrak g_{\mathbb Z}\otimes F$, called a Chevalley algebra. For a detailed account, we refer the reader to the excellent books~\cite{Ca89, Se67}. In this article, we fix the base field $F=\mathbb F_q$ of ``very good" characteristic and consider only these Lie algebras (which in the subject of Modular Lie algebras are referred to as classical, but we refer to these as Chevalley algebras). These Lie algebras $\mathfrak g_{\mathbb F_q}$ may not be simple, but are simple modulo their center. Thanks to Steinberg's work (see~\cite{St61}, and also see Chapter III~\cite{Se67}), the automorphisms of these Lie algebras are well known. These are usually the adjoint action of the corresponding Chevalley group together with diagonal and graph automorphisms (see Theorem III.5.1~\cite{Se67}).

\subsection{Lie maps}
Let $X=\{x_1, \ldots, x_k\}$ by a set and $F$ be a field. A free Lie algebra $\mathcal F_k$ on the set $X$ with $\iota \colon X \rightarrow \mathcal F_k$ is a Lie algebra which has the following universal property: For any Lie algebra $\mathfrak g$ over $F$ and a map $\chi \colon X \rightarrow \mathfrak g$ there exists a unique Lie algebra homomorphism $\tilde\chi \colon \mathcal F_k \rightarrow \mathfrak g$ such that $\chi= \tilde\chi\circ \iota$. Formally, the elements of $\mathcal F_k$ can be thought of as a sum of finitely many Lie monomials, which are in turn, finitely many steps of brackets of brackets on the elements $x_i$. For example, $e_1(x_i, x_j)=[x_i, x_j]$ and the Engel words.

An element $w\in \mathcal F_k$ over the field $F$ defines a Lie map on a Lie algebra $L$ over $F$, simply given by evaluation. Let $w(L) \subset L$ be the image of $L$ under $w$. Clearly, by taking all variables $0$, we get $0\in w(L)$. Further, for any $\psi\in \Aut(L)$, we know that $\psi(w(x_1, \ldots, x_k)) = w(\psi(x_1), \ldots, \psi(x_k))$ and hence $\psi(w(L))\subset w(L)$. Our Question~\ref{question1} asks if these properties are enough to classify the possible subset of $L$ that are the image of some $w\in \mathcal F_k$.   

When $L$ is an Abelian Lie algebra, we can answer this easily. 
\begin{proposition}
Let $L$ be an Abelian Lie algebra over a field $F$. Then, every subset $A$ of $L$ containing $0$ and closed under automorphisms is an image of some Lie polynomial.
\end{proposition}
\begin{proof}
Since $L$ is Abelian, every vector space automorphism of $L$ is also a Lie algebra automorphism. Thus, for any $v, w \in L \backslash \{0\}$, there exists $\alpha \in \Aut(L)$ such that $\alpha(v) = w$. Hence, the only possible subset of $L$ which could be an image is $\{0\}$ or $L$. These are images for the Lie polynomials $[x_1, x_2]$ and $x_1$, respectively. 
\end{proof}

\subsection{Semisimple Lie algebras}

A semisimple Lie algebra is a direct sum of simple Lie algebras or a Lie algebra that has no maximal solvable ideal. 
We state the Jordan--H\"{o}lder theorem for Lie algebras, which is required later in our work. 
\begin{theorem}[Jordan--H\"older Theorem]
Let $L$ be a Lie algebra over a field $F$. Suppose $L$ admits a finite composition series
$$
0 = L_0 \subset L_1 \subset \cdots \subset L_n = L,
$$
where each $L_i$ is an ideal of $L$ and each quotient $L_i/L_{i-1}$ is a simple Lie algebra. Then
\begin{enumerate}
\item Any two composition series of $L$ have the same length $n$.
\item The set of isomorphism classes of the simple factors $L_i/L_{i-1}$ is independent of the chosen composition series, up to permutation.
\end{enumerate}
\end{theorem}

\begin{lemma}\label{Lemma: ideals for a subalgebra of L_1 X L_2}
Let $L_1$ and $L_2$ be Lie algebras and $H \subset L_1 \oplus L_2$ be a subalgebra such that $\pi_1(H) = L_1$ and $\pi_2(H) = L_2$. Then, there exist ideals $I \trianglelefteq L_1, J \trianglelefteq L_2$ and an isomorphism $\psi: L_1/I \to L_2/J$ such that $H = \{(x, y) \in L_1 \oplus L_2 \mid \psi(x + I) = y + J \}$.  
\end{lemma}
\begin{proof}
Define $I = \{ x \in L_1 \mid (x, 0) \in H \}$ and $J = \{ y \in L_2 \mid (0, y) \in H \}$. Since $H$ is a subalgebra of $L_1 \oplus L_2$, $I$ is a subalgebra of $L_1$, and $J$ is a subalgebra of $L_2$. To show $I \trianglelefteq L_1$, let $i \in I$ and $a \in L_1$. Since $\pi_1$ and $\pi_2$ are surjective, $\exists\ b \in L_2$ such that $(a, b) \in H$. Then $[(a, b), (i, 0)] = ([a, i], 0) \in H$, so $[a, i] \in I$. Thus $I$ is an ideal. Similarly, $J \trianglelefteq L_2$.

Define $\psi\colon L_1/I \rightarrow L_2/J$ by $\psi(x + I) = y + J$ whenever $(x, y) \in H$. 
If $(x, y_1), (x, y_2) \in H$, then $(0, y_1 - y_2) \in H$, so $y_1 - y_2 \in J$, making $\psi$ well-defined. Further, $\psi$ is a Lie homomorphism because $H$ is a subalgebra. Also, $\psi$ is surjective because $\pi_1$ and $\pi_2$ are surjective. If $\psi(x + I) = J$, then $(x, y) \in H$ for some $y \in J$. Since $(0, y) \in H$, then $(x, 0) \in H$, so $x \in I$. Thus $\psi$ is injective. Hence, $\psi$ is an isomorphism and $H = \{ (x, y) \in L_1 \oplus L_2 \mid \psi(x + I) = y + J \}$.
\end{proof}

\subsection{Few generators for simple Lie algebras}\label{subsectionkL}
For our work, we need a bound on the number of generators for simple Lie algebras. For non-Abelian finite simple groups, it has been proved by Steinberg~\cite{St62} that each such group is generated by $2$ elements. In fact, a stronger result is known due to Guralnick and Kantor~\cite{GK00}, called the one-and-a-half generation. That is, for any fixed non-trivial element $x$, there exists another element $y$ such that $x, y$ generate the group. Similar questions have been studied for simple Lie algebras. 

Let $F$ be a field. When $F$ is characteristic $0$, Kuranashi~\cite{Ku51} proved that simple Lie algebras are generated by $2$ elements. Ionescu~\cite{Io76} proved that simple Lie algebras over $\mathbb {C} $ are one-and-a-half generated. When $F$ is an algebraically closed field of characteristic $0$ or $p\geq 5$, it has been proved by Bois~\cite{Bo09} that any Chevalley algebra is $2$ generated. Further, it has been proved that these are one-and-a-half generated. 

For any element $x\in L$, a simple Lie algebra over a finite field $F$, does there exist a generating set $x_1=x, x_2, \ldots, x_{k}$? If it exists, the smallest such $k$ is denoted as $k(L)$. Clearly, a vector space basis of $L$ would work as $k$, hence $k(L) \leq \dim(L)$. In fact, because of root space decomposition for a simple Lie algebra, a better bound could be taken as the number of roots. For $\Sl_n(q)$, it has been proved in~\cite{CJZ25} that this Lie algebra is $2$-generated. However, a clear answer to $k(L)$ seems not to be known.  

\section{Analog of Lubotzky's Theorem for finite simple Lie algebras}\label{Lubotzky'stheorem}

We recall that we consider $L$ a simple Chevalley Lie algebra over a field $F$ of ``very good" characteristic. For $L$ over $F$, let $k(L)$ denote the minimum number such that for each non-zero element $a \in L$, there exists elements $a_2, a_3, \dots, a_{k(L)}$ in $L$ such that the set $\{a_1=a, a_2, a_3, \dots, a_{k(L)}\}$ generates $L$. Since $k(L) \leq \dim(L)$, it is a finite number. For a brief discussion on this, see Section~\ref{subsectionkL}. Similar to the case of finite simple groups as in~\cite{Lu14}, this number would turn out to be the number of variables required to get the required subset as an image. 

Now, we are ready to prove our main theorem. 
\begin{theorem}\label{Thm: Image canditates are images}
Let $L$ be a simple Chevalley algebra over a finite field $F$ of ``very good'' characteristic. Then, every subset $A\subset L$ containing $0$ and closed under automorphisms of $L$ is an image of some Lie polynomial from $\mathcal F_{k(L)}$.
\end{theorem}
The proof will be in several steps. We will consider a non-Abelian Lie algebra. It closely follows the proof of a similar result by Lubotzky in~\cite{Lu14} for non-Abelian finite simple groups. We begin by setting up the machinery for the same.

Take $k=k(L)$ and consider the free Lie algebra $\mathcal{F}_k$ on the generating set $x_1, \ldots, x_k$. Let $\mathcal{I}$ be an indexing set having the same cardinality as the set $L^k$ (finiteness of $L$ is crucial here). Let 
$$\mathcal C= \{(g_{(i,1)}, g_{(i,2)}, \ldots, g_{(i,k)} ) \mid i \in \mathcal{I}\}$$
be the set of all ordered $k$-tuples of elements of $L$. We write the indexing set $\mathcal{I} = \mathcal{I}_1 \cup \mathcal{I}_2$ such that $i \in \mathcal{I}_1$ if and only if the tuple $(g_{(i,1)}, g_{(i,2)}, \dots, g_{(i,k)} )$ is a generating set of $L$. Note that for every $i \in \mathcal{I}$, the corresponding tuple $(g_{(i,1)}, g_{(i,2)}, \ldots, g_{(i,k)} )$ induces a unique Lie algebra homomorphism on the free Lie algebra $\phi_i \colon \mathcal{F}_k \to L$ defined by $\phi_i(x_j) = g_{(i,j)}$, for $1\leq j \leq k$. Clearly, when $i \in \mathcal{I}_1$, the map $\phi_i$ is surjective and hence $\mathcal{F}_k / N_i \cong L$. If $i \in \mathcal{I}_2$, then $\phi_i$ is not surjective and $\mathcal{F}_k / N_i$ is isomorphic to a proper Lie subalgebra of $L$. With this in mind let $N_i := \rm Ker(\phi_i)$, $M := \bigcap_{i \in \mathcal{I}} N_i$ and $N := \bigcap_{i \in \mathcal{I}_1} N_i$. 

Now, consider the diagonal map given by,
\begin{eqnarray*}
\phi_M = (\phi_i)_{i \in \mathcal{I}}&\colon & \mathcal{F}_k \rightarrow \displaystyle \prod_{i \in \mathcal{I}} L\\ 
&& x_j \mapsto \big(\phi_i(x_j)\big)_{i \in \mathcal{I}} = (g_{(i,j)})_{i \in \mathcal{I}} 
\end{eqnarray*}
and denote $H:= \phi_M (\mathcal{F}_k) \subset \prod_{i \in \mathcal{I}} L$. We write $\displaystyle \prod_{i \in \mathcal{I}} L = E \times D$, where $E =\displaystyle \prod_{i \in \mathcal{I}_1} L$ and $D = \displaystyle \prod_{i \in \mathcal{I}_2} L$. Consider the natural projection maps, $\pi_E\colon H \to E$ and $\pi_D \colon H \to D$. Let $K_1:= \rm Ker(\pi_E)$, $E':= \rm Image(\pi_E)$, $K_2:= \rm Ker(\pi_D)$, $D':= \rm Image(\pi_D)$. Now, we aim to understand the structure of these ideals and subalgebras. 

\begin{lemma}\label{lemma:Jordan-Holder}
With the notation as above, $K_1 \subset H$ has no Jordan--H\"older factor isomorphic to $L$.
\end{lemma}
\begin{proof}
We have $K_1 = \{(\phi_i(w))_{i \in \mathcal{I}} \mid w \in \mathcal{F}_k \text{ and } \phi_i(w) = 0, \text{ for } i \in \mathcal{I}_1\}\subset H \subset E\times D$. Thus, $K_1$ is isomorphic to a Lie subalgebra of $D$. Note that the projection of $K_1$ to every single copy of $L$ in $D$ is a proper Lie subalgebra of $L$. Thus $K_1$ has no Jordan--H\"older factor isomorphic to $L$.
\end{proof}


Now, we aim to understand the structure of $E'$.
\begin{lemma}\label{lemma:EJ}
With the notation as above, $E':=\pi_E(H)$ is isomorphic to $\displaystyle \prod_{i \in \mathcal{J}} L$ for some indexing set $\mathcal{J} \subseteq \mathcal{I}_1$.
\end{lemma}
\begin{proof}
Recall, we have $\pi_E\colon H \to E$ and $E' = \rm \pi_E (H) \cong H/K_1$. Note that $E'$ is isomorphic to a subalgebra of $E$ whose projection to each copy of $L$ in $E$ is $L$. We need to show that $E'$ is isomorphic to $\displaystyle \prod_{i \in \mathcal{J}} L$ for some indexing set $\mathcal{J} \subset \mathcal{I}_1$. Let $f_j\colon E' \to E'$ be a map defined as $f_j \big((\phi_i(w))_{i \in \mathcal{I}_1}\big) = (a_i)_{i \in \mathcal{I}_1}$, where $a_i = \phi_i(w)$ for $i \neq j$ and $a_j = 0$. Observe that $\text{Ker}(f_j)= \{\big (\phi_i(w) \big)_{i \in \mathcal{I}_1} \mid  w\in \mathcal{F}_k, \phi_i(w) = 0, \text{ for } i \neq j\} = \prod_{i \in \mathcal{I}_1} U_i$, where $U_i =0$, if $i \neq j$ and $U_j = \{\phi_j(w) \mid w \in \cap_{\mathcal{I}_1 \backslash \{j\}} N_i\}$. Since $\text{Ker}(f_j)$ is an ideal of $E'$, it is an ideal of $\prod_{i \in \mathcal{I}_1} V_j$, where $V_i =0$, if $i \neq j$ and $V_j =L$. Thus, by the fact that $L$ is simple, $\text{Ker}(f_j)$ is either isomorphic to $\{0\}$ or $L$ for each $j$. Now, we take $\mathcal{J}:= \{j \in \mathcal{I}_1 \mid \text{Ker} (f_j) \cong L\}$. Let $j \in \mathcal{I}_1\backslash \mathcal{J}$. Then $\text{Ker} (f_j) \cong \{0\}$ and hence, $E' \cong E'/ \text{Ker} (f_j) \cong \img (f_j)$. Further, note that for $j \in \mathcal{I}_1 \backslash \mathcal{J}$, $\img (f_j)$ is isomorphic to the image of the projection map $\pi_E$ when $\mathcal{I}_1$ is replaced by $\mathcal{I}_1 \backslash \{j\}$. Thus, we can assume $\mathcal{I}_1 = \mathcal{J}$ and hence $E' \subset \displaystyle \prod_{i \in \mathcal{J}} L$. With this assumption, let $ (a_i)_{i \in \mathcal{J}}$ be an element of $\displaystyle \prod_{i \in \mathcal{J}} L$. By the definition of $\mathcal{J}$, for each $j \in \mathcal{J}$, there exists $w_j \in \mathcal{F}_k$ such that $\phi_i(w_j) = 0$ if $i \in \mathcal{I}_1 \backslash {j}$ and $\phi_j(w_j) = a_j$. We choose $w = \sum_{j \in \mathcal{J}} w_j \in \mathcal{F}_k$ and $\pi_E(\phi_M(w)) = (a_i)_{i \in \mathcal{J}}$. Note that $\pi_E(\phi_M(w)) \in E'$ and hence $\displaystyle \prod_{i \in \mathcal{J}} L \subset E'$ which establishes $\displaystyle \prod_{i \in \mathcal{J}} L \cong E'$. 
\end{proof}

We now proceed further to compute $|\mathcal{J}|$.  We have, $\Aut(L)$ acts on the set $\mathcal C = (g_{(i,1)}, \dots, g_{(i,k)})$ of $k$-tuples in $L$ by $\psi.(g_{(i,1)}, \dots, g_{(i,k)}) = (\psi(g_{(i,1)}), \dots, \psi(g_{(i,k)}))$ where $\psi\in \Aut(L)$. This is a free action because, if $\psi \in \Aut(G)$ and 
$$\psi \big( (g_{(i,1)}, g_{(i,2)}, \dots, g_{(i,k)} \big) = \big(\psi (g_{(i,1)}), \psi (g_{(i,2)}), \dots, \psi (g_{(i,k)}) \big) = \big( (g_{(i,1)}, g_{(i,2)}, \dots, g_{(i,k)} \big)$$ then $\psi$ maps a generating set of $L$ identically to itself and hence $\psi$ is identity. We can further restrict this action to the subset of $\mathcal C$ with index $i\in \mathcal I_1$. Note that we have an equivalent action of $\Aut(L)$ on the set of homomorphisms $\{\phi_i \colon \mathcal F_k \rightarrow L \mid i \in \mathcal{I}\}$ given by $\psi.\phi_i = \psi\circ \phi_i$ which preserves the set $\{\phi_i \mid i \in \mathcal{I}_1\}$. Thus, the set of homomorphisms $\{\phi_i \mid i \in \mathcal{I}_1\}$ forms $r = \frac{|\mathcal{I}_1|}{|Aut(L)|}$ orbits. Let us denote these $r$ orbits be $\phi_{i_1}, \phi_{i_2}, \dots, \phi_{i_r}$. Some of the ideas in the proof of the following Lemma are inspired by Proposition 6 and Corollary 7 in~\cite {KL90}.

\begin{example} Note that a pair of generating sets need not be in the same automorphism orbit. Take the following generating sets $S_1=\{e, f\}$ and $S_2=\{h, e+f\}$ for the Lie algebra $L=\Sl_2(q)$. There doesn't exist an automorphism of $L$ that maps $S_1$ to $S_2$. This is because $S_2$ has one element that is semisimple, whereas $S_1$ has both nilpotent elements.
\end{example}

\begin{lemma}\label{lemma:E'Lr}
With the notation as above, $E' \cong L^r$. 
\end{lemma}
\begin{proof}
From Lemma~\ref{lemma:EJ}, we have $E'\cong L^{|\mathcal J|}$. We need to show that $|\mathcal{J}| = r$. For that, we let 
$$T = \{ (u_i)_{i \in \mathcal{I}_1} \in L^{|\mathcal{I}_1|} \mid \text{ If } \psi \circ \phi_i = \phi_j \text{ for some } \psi \in \Aut(L), \text{ then } \psi(u_i) = u_j \}.$$ 
By the definition of $T$, any element $(u_i)_{i \in \mathcal{I}_1} \in T$ is uniquely determined by the $r$-tuple $(u_{i_1}, \dots, u_{i_r})$ because all other coordinates are fixed by the automorphisms. Therefore, $T \cong L^r$. We claim that $T = E'$. 
Let $(\phi_i(w))_{i \in \mathcal{I}_1} \in E'$. If for some $\psi \in \Aut(L)$, $\psi \circ \phi_i = \phi_j$, then $\psi(\phi_i(w)) = \psi \circ \phi_i (w) = \phi_j (w)$, and hence $(\phi_i(w))_{i \in \mathcal{I}_1} \in T \implies E' \subset T$. 

For $T \subset E'$, we proceed by induction on $r$ to prove that the map $\mathcal{F}_k \to L^r$, $w \mapsto \big(\phi_{i_1}(w), \phi_{i_2}(w), \dots , \phi_{i_r}(w)\big)$ is surjective. Since each $\phi_i$ is surjective for $i \in \mathcal{I}_1$, the case $r=1$ is trivial. Assume surjectivity for $r-1$. Let $V$ be the image of $\mathcal{F}_k \rightarrow L^r$. Then by induction hypothesis for $r-1$, we have $V \subset L^{r-1} \oplus L_r$, where $L_r:= \{\phi_{i_r}(w) \mid w \in \mathcal{F}_k\} = L$. 
By Lemma~\ref{Lemma: ideals for a subalgebra of L_1 X L_2}, there exists ideals $I_r$ and $J_r$ such that $L^{r-1}/I_r \cong L_r/J_r$, where $I_r = \{(h_1,h_2, \dots, h_{r-1}) \in L^{r-1} \mid (h_1, h_2, \dots, h_{r-1},0) \in V\}$ and $J_r = \{h_r \in L_r \mid (0, \dots, 0, h_r) \in V\}$. Since $L_r$ is simple, $J_r$ is either $\{0\}$ or $L_r$. If $J_r = L_r$, let $(y_1, y_2, \dots, y_r)$ be an element of $L^r$. Then, by surjectivity for $r-1$ components, there exist $w \in \mathcal{F}_k$ such that $\big(\phi_{i_1}(w), \phi_{i_2}(w), \dots , \phi_{i_r}(w)\big) \in V$ and $y_j = \phi_{i_j}(w)$ for $j \in \{1, 2, \dots, j-1\}$. Since $J_r =L_r$, $(0,\dots, 0, y_r-\phi_{i_r}(w)) \in V$ and thus $(y_1, y_2, \dots, y_r) \in V$. Evidently, $V = L^r$ and $V$ is the maximum quotient of $\mathcal{F}_k$ which is isomorphic to a direct power of $L$. Thus, $L^r=V \cong E'$ which leads to $E' = T$. 

Now, let $J_r = \{0\}$. Note that for any $m \leq r$, we can define $L_m:= \{\phi_{i_j}(w)\mid  w \in \mathcal{F}_k\}$ and apply Lemma~\ref{Lemma: ideals for a subalgebra of L_1 X L_2} to similarly get ideals $I_m$ and $J_m$. If $J_m = L_m$ for any $m$, then by a similar approach, we have $E' = T$. Thus, we can assume that $J_m = \{h_m \in L_m \mid (0, \dots, 0, h_m, 0, \dots, 0) \in V\} = \{0\}$ for all $m$. Since $I_r$ is an ideal of $L^{r-1}$, $J_m= \{0\}$ for all $m$ and $L$ is non-Abelian and simple, we have $I_r = \{0\}$. Hence, the projection $\pi_r\colon V \rightarrow L_r$ is an isomorphism. Similarly, for any $j$, the projection $\pi_j\colon V \to L_j$ is an isomorphism. Hence, for any $j,m \in \{1,2,\dots,r\}$, we get an isomorphism $L_m \to V \to L_j$, sending $\phi_{i_m}(w) \mapsto (\phi_{i_1}(w), \dots, \phi_{i_r}(w)) \mapsto \phi_{i_j}(w)$. This mean there exists an isomorphism $\sigma\colon L_m \rightarrow L_j$ such that for every $w \in \mathcal{F}_k$, $\sigma(\phi_{i_m}(w)) = \phi_{i_j}(w) \implies \sigma \circ \phi_{i_m} = \phi_{i_j}$. Since $\sigma$ is an isomorphism between $L$ and itself, $\sigma \in \Aut(L)$, implying that $\phi_{i_m}$ and $\phi_{i_j}$ have the same orbit, which contradicts our choice of distinct orbits. Thus, $L^r=V \cong E' = T$.
\end{proof}

\begin{lemma}\label{lemma:HE'D'}
With the notation as above, $H \cong E' \times D'$.
\end{lemma}
\begin{proof}
Now, observe that $K_2= \rm{Ker} (\pi_D) =  \{(\phi_i(w))_{i \in \mathcal{I}} \mid  w \in \mathcal{F}_k \text{ and } \phi_i(w) = 0, \text{ for } i \in \mathcal{I}_2\}$ is isomorphic to a Lie subalgebra of $E'$. Further, $H/K_2 \cong D'$, which has no Jordan--H\"{o}lder factors isomorphic to $L$. So, all the Jordan--H\"{o}lder factors of $H$ which are isomorphic to $L$ should appear in $K_2$ only. Since $E' \cong \displaystyle \prod_{i \in \mathcal{J}} L$ is a projection of $H$ and $K_2$ is isomorphic to a Lie subalgebra of $E'$, $K_2 = E' \times \prod_{i \in \mathcal{I}_2} \{0\} \implies E' \times \prod_{i \in \mathcal{I}_2} \{0\} \subset H$. Further, every element of $H$ is of the form $(a,b)$, where $a \in E'$ and $b \in D'$. Thus, if $E' \times \prod_{i \in \mathcal{I}_2} \{0\} \subseteq H$, then $\prod_{i \in \mathcal{I}_1} \{0\} \times D' \subset H$, and hence $H \cong E' \times D'$.
\end{proof}

Now, we complete the proof of the theorem.
\begin{proof}[{\bf Proof of Theorem~\ref{Thm: Image canditates are images}}]

Let $A \subset L$ containing $0$ and closed under automorphisms of $L$. We denote $A' = A \backslash \{0\}$. Because of the existence of $k(L)$, each element of $A'$ can be extended to a generating set of $L$. Hence for each $a \in A'$, there exists $a_2, a_3, \ldots, a_{k}$ such that $(a, a_2, \ldots, a_k) = (g_{(i,1)}, g_{(i,2)}, \dots, g_{(i,k)})$ for some $i \in \mathcal{I}_1$. There is a natural action of $\Aut(L)$ on $L$, and $A'$ is a union of certain orbits under this action. Additionally, $\Aut(L)$ also acts on the set of $k$-tuples $(g_{(i,1)}, g_{(i,2)}, \dots, g_{(i,k)})$ in $L$, and hence on the set of homomorphisms $\{\phi_i\mid  i \in \mathcal{I}\}$. This action preserves the set $\{\phi_i\mid  i \in \mathcal{I}_1\}$. The orbit of any element in $L$ under the action of $\Aut(L)$ on $L$ gives an orbit of $k$-tuples containing the given element in the first coordinate. We define the following element $\bar{z} = (z_i)_{i \in \mathcal{I}}$ of $\displaystyle \prod_{i \in \mathcal{I}} L$ as follows
$$z_i =
\begin{cases}
a_i = g_{(i,1)}, & \text{if } i \in \mathcal{I}_1 \text{ and } g_{(i,1)} \in A',\\
0,   & \text{otherwise}.
\end{cases}
$$
We claim that $\bar{z} \in H = E' \times D'$. By definition, the projection of $\bar{z}$ to $D$ is zero. So, we just need to show that its projection to $E$ lies in $E'$. Note that an element $(u_i)_{i \in \mathcal{I}_1} \in E'$ if and only if whenever $\psi \circ \phi_i = \phi_j$ for some $i \in \mathcal{I}_1$ and $\psi \in \Aut(L)$, $\psi(u_i) = u_j$. So, $E'$ is embedded diagonally in $E$ twisted by $\Aut(L)$. Let $i,j \in \mathcal{I}_1$ be such that $\psi \circ \phi_i = \phi_j$. Then $\psi(a_i) = a_j$. We need to show that $\psi(z_i) = z_j$. If $a_i \notin A'$, then $z_i =0$ gives $\psi(z_i) = 0$. Since $A'$ is $\Aut(L)$-invariant, $\psi(a_i) = a_j \notin A' \implies z_j=0 =\psi(z_i)$. If $a_i \in A'$, then $z_i =a_i \implies \psi(z_i) =\psi(a_i) = a_j$. Since $A'$ is $\Aut(L)$-invariant, $a_j \in A' \implies z_j=a_j = \psi(z_i)$. Hence, $\pi_E(\bar{z}) \in E'$ and thus $\bar{z} \in H$. Moreover, each element of $A'$ can be extended to a generating set of $L$ and is thus equal to some coordinate of $\bar{z}$. 

The above description shows that $\bar{z} \in H$ and hence there exists $w \in \mathcal{F}_k$ such that $\phi_M(w) = \bar{z}$. Consequently, for each $i \in \mathcal{I}$, $z_i= \phi_i(w) = w(g_{(i,1)}, g_{(i,2)}, \dots , g_{(i,k)})$. This shows that for this $w \in \mathcal{F}_k$, $w(L) = \{z_i\mid i \in \mathcal{I}\}= A' \cup \{0\} = A$.
\end{proof}

\section{Some Examples}\label{examples}
In this section, we present examples to illustrate the strength and limitations of the Theorem proved earlier. 

\begin{example}
The conclusion of Theorem \ref{Thm: Image canditates are images} may not hold if $L$ is not a simple Lie algebra. For example, if $L = \Sl_2(q) \oplus \Sl_2(q)$, then the automorphism orbits of $L$ are of the type $(O_1, O_2) \cup (O_2,O_1)$, where $O_1$ and $O_2$ are automorphism orbits of $\Sl_2(q)$. Thus, if $O_1 \neq O_2$, then for any Lie polynomial $w$, $w(L) = w(\Sl_2(q) \oplus \Sl_2(q)) = w(\Sl_2(q)) \oplus w(\Sl_2(q)) \ne (O_1, O_2) \cup (O_2,O_1) \cup (0,0)$. Hence, $(O_1, O_2) \cup (O_2, O_1) \cup (0,0)$ satisfies the two required criteria in the theorem to be an image, but is not an image.
\end{example}
\begin{example}
In~\cite{BGKP12}, examples of Lie polynomials with dominant image are given in Remark 3.6, 3.7, and Example 3.8. Further, taking the sum of images of two dominant morphisms will give surjectivity. Using this, the authors show surjectivity of Engel polynomials under some mild hypotheses.
\end{example}

\begin{example}
Some examples of Lie polynomials giving particular images are listed in Section 2.2, Example 1 of~\cite{KMR17}. The authors also show that (see Theorem 3) over an algebraically closed field of characteristic $\neq 2$, the image of any Lie polynomial on $\Sl_2$ is one of the following: $\Sl_2$, $\{0\}$, or the set of non-nilpotent elements in $\Sl_2$. Clearly, our main Theorem fails in this case.
\end{example}

The following proposition establishes that the conclusion of Theorem \ref{Thm: Image canditates are images} does not hold if $L$ is a non-abelian, non-perfect Lie algebra.

\begin{proposition}
Let $L$ be a Lie algebra over a field $F$ such that $0 \neq [L,L] \neq L$. Then $L$ has a subset which is not an image of any word.  
\end{proposition}
\begin{proof}
Let $\mathcal{F}_k= \langle x_1, x_2, \dots, x_k \rangle$ be the free Lie algebra and $w \in \mathcal{F}_k$ be a Lie polynomial. Then $w = w' + \sum_{i=1}^k a_i x_i$, where $w' \in [\mathcal{F}_k, \mathcal{F}_k]$ and $a_i \in F$. If $(a_1, a_2, \ldots, a_k) \neq 0$, then $w(L) =L$, and if $(a_1, a_2, \ldots, a_k) = 0$, then $w = w'$ and hence $w(L) = w'(L) \subset [L, L]$. Evidently, if $A \subset L \backslash [L,L]$ is an automorphism invariant subset of $L$, then $A \cup \{0\}$ satisfies the properties to be an image set which can not be the image for any Lie polynomial. In particular, $L \backslash [L,L] \cup \{0\}$ is a subset that contains $0$ and is closed under automorphisms but is not an image of any word on $L$.
\end{proof}

\begin{example}
We can apply the above Proposition to $\mathfrak gl_n(q)$ as $[\mathfrak gl_n(q), \mathfrak gl_n(q)] =\Sl_n(q)$. Hence, the set containing $0$ and a non-trivial element of the center can not be the image of any Lie polynomial. 
\end{example}


\begin{example}
Let $w(x_1, \ldots, x_{k_1}) \in \mathcal F_{k_1}$ and $u(y_1, \ldots, y_{k_2}) \in \mathcal F_{k_2}$ then the image of $w+u$ on a Lie algebra is $w(L) + u(L)$. In~\cite{BGKP12}, this is used to prove surjectivity of Engel polynomials in certain cases. 
\end{example}

\section{Images of polynomial maps on $\Sl_2(q)$}\label{image1}
For a given subset $A$ of $L$ having $0$ and closed under automorphism, the theorem proved earlier doesn't explicitly give a word that gives $A$ as its image. In the remaining part of this article, we give examples of words on $L=\Sl_2(q)$ that can give such subsets as the image.

\subsection{$\rm{GL}_2(\mathbb F_q)$-orbits in $\Sl_2(q)$} 
Let $\mathbb{F}_q$ be a finite field of size $q$ for some $q = p^r$, where $r >0$ and $p$ is an odd prime. The only Lie algebra automorphisms for the Lie algebra $\Sl_2(q)$ are the automorphisms defined by the conjugation action of the group $\rm{GL}_2(\F_q)$ on the Lie algebra $\Sl_2(q)$ (see~\cite{St61}). Also, since the scaler matrices act trivially, the group of Lie algebra automorphisms for the Lie algebra $\Sl_2(q)$ is isomorphic to $\rm{PGL}_2(\F_q)$. Additionally, the characteristic polynomial for an element $A\in \Sl_2(q)$ is $X^2 + \det(A)$, thus $A^2=-\det(A)$. Hence, the nonzero orbits under the above action of $\Sl_2(q)$ are characterized by the determinant values. All non-zero nilpotent matrices of $\Sl_2(q)$ have zero determinant, and they form a single orbit while the semisimple orbits are in one-to-one correspondence with the non-zero determinant values. Let $\left\{e = \begin{pmatrix} 0&1\\0&0\end{pmatrix}, f = \begin{pmatrix} 0&0 \\ 1&0\end{pmatrix}, h= \begin{pmatrix} 1&0\\0&-1\end{pmatrix}\right \}$ be the standard basis of $\Sl_2(q)$. The Lie algebra relations are $[h,e]=2e, [h,f]=-2f, [e,f]=h$. Then note that $\det(e + af) = -a$ for any $a \in \F_q$. Thus, the set $\{e + af \mid  a\in \F_q\} \cup \{0\}$ is a set of orbit representatives for the automorphism action on $\Sl_2(q)$. The orbits of the action are listed in the Table~\ref{tab:orbits}. 

\begin{table}[]
   \centering
    \begin{tabular}{|l|c|c|c|c|}\hline
orbit type & & Representative & determinant & no of orbits  \\ &&&&\\ \hline
Zero  &$\mathcal O$ & 0 & 0&1 \\\hline
Nilpotent &$\mathcal O_1$& $e=\begin{pmatrix} 0&1\\ 0&0 \end{pmatrix}$ & $\det = 0$ & 1\\ \hline
Split Semisimple& $\mathcal O_2$& $e+af \sim \begin{pmatrix} \lambda & \\ &-\lambda \end{pmatrix}$ & $\det =-a= -\lambda^2$, & $\frac{q-1}{2}$ \\
& & & $a=\lambda^2 \in  {\mathbb F_q^*}^2$& \\ \hline
Anisotropic Semisimple & $\mathcal O_3$ & $e+af \sim \begin{pmatrix}  & a \\ 1&  \end{pmatrix}$ & $\det = -a, a \not\in {\mathbb F_q^*}^2$& $\frac{q-1}{2}$\\ \hline 
\end{tabular}
\caption{$\rm{GL}_2(\mathbb F_q)$-orbits in $\Sl_2(q)$} \label{tab:orbits}
\end{table}

The total number of orbits is $q+1$, characterized by the determinant of the matrix, except for determinants zero, where these belong to two different types. We further note that $q\equiv 1 \mod 4$ if and only if $-1 \in {\mathbb F_q^*}^2$. In that case, $\det(A) \in {\mathbb F_q^*}^2$ (and $\det(A) \not\in {\mathbb F_q^*}^2$ ), if and only if the element is split semisimple (and anisotropic semisimple).
 
\subsection{Working Technique} Let $w \in \mathcal F_2$ be a Lie polynomial and $L$ be a Lie algebra. Let $\mathcal O_L$ be a set of orbit representatives for the automorphism action on $L$. Then, to find the image of $w$ on $L$, we follow the steps below.
\begin{enumerate}
\item Find the set $w(L, \mathcal O_L):= \{w(x,y) \mid  x \in L, y \in \mathcal O_L\}$.
\item Let $\mathcal O_{w(L)}:= \{x \in \mathcal O_L \mid \psi(x) \in w(L, \mathcal O_L)\text{ for some } \psi \in \Aut(L)\}$. 
\item Obtain the set $\bigcup \mathcal O_{w(L)}$. 
\end{enumerate}

We claim that the set $\bigcup \mathcal O_{w(L)}$ obtained above is equal to the image set $w(L)$. For this, first observe that $\bigcup \mathcal O_{w(L)} \subset w(L)$ since $w(L)$ is an automorphism invariant subset of $L$. Also, if $x, y \in L$ and $y' \in L$ be such that there exist an automorphism $\psi$ of $L$ with $\psi(y')=y$, then $w(x, y) = w(x, \psi(y'))= \psi \big(w(\psi^{-1}(x), y')\big)$. Thus, the set $\mathcal O_{w(L)}$ contains representatives of all orbits that constitute the set $w(L)$.

\begin{remark}\label{one variable over representatives}
If we fix $L = \Sl_2(q)$, then as discussed above, we can choose $\mathcal O_L$ as the $\{e + af \mid a \in \F_q\} \cup \{0\}$. Also, to calculate $w(L)$ by the above technique, we lose nothing by replacing $\mathcal O_L$ with $\mathcal O_L\backslash\{0\}$. So, to find the set $w(L): = \{w(x, y)\mid x, y \in L \}$, we first obtain the set $w(L, \mathcal O(L)= \{w(x, y) \mid x \in \Sl_2(q), y = e + af \text{ for some }a \in \F_q \}$, and then find the set of possible determinant values for the elements of $w(L, \mathcal O(L))$ to figure out the step (2). 
\end{remark}

\subsection{$ad$-map on $\Sl_2(q)$} 
Let $A, X\in \Sl_2(q)$. We denote the map $ad_A\colon \Sl_2(q) \rightarrow \Sl_2(q)$ given by $ad_A(X) = AX-XA$, and more generally $ad_A^n(X) = ad_A(ad_A^{n-1}(X))$ which are nothing but the Engel words. Further, note that ${ad_A(X)}^n$ would be a totally different thing, not to be confused with the above notation. Our aim to find Lie polynomials that give their image the union of one or more orbits, together with the zero orbit, will involve these maps. 

Recall that for any $A\in \Sl_2(q)$, we have $A^2=-d$, a scalar, where $d=\det(A)\in \mathbb F_q$. We have the following:
\begin{lemma}
$Aad_A(X)A = -A^2ad_A(X)$. 
\end{lemma}
\begin{proof}
This is a simple computation using $A^2$, which is a scalar and thus commutes. Check, $Aad_A(X)A = A(AX-XA)A = A^2XA - AXA^2 = A^2(XA-AX) = -A^2 ad_A(X)$. 
\end{proof}

\begin{lemma}\label{higher ad maps}
For $n\geq 1$, $ad_A^n(X) = 2^{n-1}A^{n-1} ad_A(X)$.
\end{lemma}
\begin{proof}
The proof is simply by induction. Clearly, it is true for $n=1$. Let us assume this for $n$ and show for $n+1$. We have,
$ad_A^{(n+1)}(X) = ad_A(ad_A^n(X)) = ad_A(2^{n-1}A^{n-1} ad_A(X)) = 2^{n-1}(A^n ad_A(X) - A^{n-1}ad_A(X)A) = 2^{n-1}(A^n ad_A(X) + A^n ad_A(X)) = 2^nA^n ad_A(X)$. 
\end{proof}
In what follows, we frequently use the fact that every element $\alpha\in \mathbb F_q$ satisfies $\alpha^q = \alpha$, and for $\alpha\neq 0$, we have $\alpha^{q-1}=1$. The same will also hold for any diagonal matrix in $M_2(\mathbb F_q)$.

\section{Sum of Engel words on $\Sl_2(q)$}\label{image2}

Let us consider the Lie polynomial $f(A, X) = ad_A^i(X)- ad_A^{j}(X)$ with $i, j \geq 1$ and determine its image in certain cases. Note that when $A=0$, we get $0$ in the image. So, while doing computations, we may assume $A\neq 0$. 

\begin{proposition}\label{prop:nilotent}
Suppose $q$ is odd. Then the set of all nilpotent elements (including zero) is obtained as the image of the Lie polynomial $ad_A^2(X) - ad_A^{2q}(X)$. 
\end{proposition}
\begin{proof}
We note that $f(A,X) = ad_A^2(X) - ad_A^{2q}(X) = (2A - 2^{2q-1}A^{2q-1}) ad_A(X) = 2A(1 - 2^{2(q-1)}A^{2(q-1)}) ad_A(X)$. Now, $2^{2(q-1)}=1$, and 
$A^{2(q-1)} = (-\det(A))^{q-1} = \begin{cases}
 1 & \text{ if} \det(A)\neq 0 \\ 0 &  \text{otherwise}. \end{cases}$
Thus, 
$$f(A,X) = \begin{cases} 0 & \text{when} \det(A) \neq 0 \\ 2A ad_A(X) & \text{when} \det(A)=0 \end{cases}.$$
when $A$ has determinant $\neq 0$, i.e., $A$ is semisimple, we get $0$ as image for any such $A$. If $A$ is nilpotent, we claim that the image is always nilpotent and non-trivial. 

Checking nilpotent is easy by showing the square is $0$. For this, take $(2Aad_A(X))^2 = 4 Aad_A(X)A ad_A(X) = 4 A^2 ad_A(X)ad_A(X) =0$ since $A$ is nilpotent. Take $A=e$, $X=uh+ve+rf$, then $ad_A(X) =[e, uh+ve+rf] = -2ue + rh$, and $ad_A^2(X) =[e, -2ue + rh] = -2re$. This completes the proof.
\end{proof}

\noindent Now, we obtain nilpotent union all split semisimple orbits as the image. 
\begin{proposition}
Suppose $q$ is odd. Then the set of all nilpotent and all split semisimple elements (together with zero) is the image of the Lie polynomial $ad_A(X) - ad_A^{2q-1}(X)$. 
\end{proposition}
\begin{proof}
We note that $f(A, X) = ad_A(X) - ad_A^{2q-1}(X) = (1 - 2^{2q-2}A^{2q-2}) ad_A(X)$. Since $2^{2(q-1)}=1$ and $A^{2(q-1)} = (-\det(A))^{q-1} = \begin{cases} 1 & \text{ if} \det(A)\neq 0 \\ 0 & \text{otherwise},
\end{cases}$ we get the following:
$$f(A,X) = \begin{cases} 0 & {\text when} \det(A) \neq 0 \\ ad_A(X) & {\text when}  \det(A)=0. \end{cases}$$
Thus, when $A$ is semisimple, no matter what $X$ is, we get $0$ as the image of $f(A, X)$. We claim that when $A$ is nilpotent, $ad_A(X)$ is either nilpotent or split semisimple. In fact, both nilpotent and split semisimple elements are present in the image of $f(A, X)$ in that case. 

Since $\det(A) =0$, it is enough to assume $A = ae$ and $X = b_1e + b_2 f + b_3h$, for some $a, b_1, b_2, b_3 \in \F_q$. With these assumptions, $ad_A(X)= [ae, b_1e + b_2 f + b_3h] = ab_2h - 2 a b_3e$ which has determinant equal to $-{(ab_2)}^2$. This is zero when $ab_2 =0$ and thus in that case $ad_A(X)= -2ab_3e$. Hence, by taking $b_2=0$, we see that the nilpotent orbit is in the image. Further, note that if $-{(ab_2)}^2 \ne 0$, then $ad_A(X)$ is a split semisimple element with characteristic polynomial $Y^2 - {(ab_2)}^2$. This completes the proof.
\end{proof}

\noindent Now we show that all semisimple elements, together with $0$, can be obtained as images.
\begin{proposition}\label{prop: all semisimple orbits}
Suppose $q$ is odd. Then, the image of the Lie polynomial $ad_A^{q+1}(X) - ad_A^{2q}(X)$ contains all semisimple elements (of both types) together with zero. 
\end{proposition}
\begin{proof}
We consider the Lie polynomial  $f(A,X) = ad_A^{q+1}(X) - ad_A^{2q}(X) = 2^qA^q ad_A(X) - 2^{2q-1}A^{2q-1} ad_A(X) = 2A(2^{q-1}A^{q-1} - 2^{2(q-1)}A^{2(q-1)})ad_A(X)$. Now, consider that $\mathbb F_q^*$ is cyclic of order $q-1$, and $A^2 = -d$ where $d=\det(A)$, we get
$$f(A,X) = \begin{cases}2A(A^{q-1}-1) ad_A(X) & \text{when}\ d\neq 0 \\ 2A^q ad_A(X) =0 & \text{when} \ d=0 \end{cases}.$$
When $d=0$, we have $A^2=0$, thus $f(A,X)=0$, as noted above. Further, when $A=\begin{pmatrix} \lambda & \\ & -\lambda\end{pmatrix}$ is split semisimple, we have $A^{q-1} =1$. Thus, $f(A, X) = 0$ in this case too. We note that a possible non-zero element in the image of $f(A, X)$ arises when $A$ is anisotropic semisimple.  

Let $A=e+af$ with $a\not\in {\mathbb F_q^*}^2$. Then, $A^2 = a$ and $A^{q-1} = a^{\frac{q-1}{2}}$. Now, $f(A,X) = 2(a^{\frac{q-1}{2}}- 1)A ad_A(X)$. Since $a \not\in {\mathbb F_q^*}^2$, $a^{\frac{q-1}{2}} \neq 1$. Let $X=uh + ve + rf$, then $ad_A(X) = [e+af, uh + ve +rf] = -2ue +(r-av)h + 2auf$ with determinant $= -(r-av)^2 + 4au^2$. Thus, $\det(Aad_A(X)) = -a(-(r-av)^2 + 4au^2)$. Now, we claim that $\det(Aad_A(X))\neq 0$; otherwise, if $\det(Aad_A(X))= 0$, we get $-(r-av)^2 + 4au^2 = 0$. If $v=0$, we get $4au^2=r^2$, a contradiction that $a$ is a non-square. In general solving for $a$ we get $a= \frac{(u\pm \sqrt{u^2+rv})^2}{v^2}$ which is a contradiction that $a$ is a non-square. Hence, the image has no nilpotents as $\det(Aad_A(X))\neq 0$. This shows that the non-zero elements of image are contained in the subset of semisimple elements. 

Now check that $\det(Aad_A(X)) = a(r-av)^2 - 4a^2u^2$ can take all values over $\mathbb F_q$. If $-1$ is a square, let $i \in \mathbb F_q$ be such that $i^2=-1$. Then, $\det(Aad_A(X)) = a(r-av)^2 - 4a^2u^2 = a(r-av)^2 + (2iau)^2$. Since $a\not\in {\mathbb F_q^*}^2$, by taking $u=0$, $\det(Aad_A(X)) = a(r-av)^2$ gives all non-square values, and by taking $r=0=v$, we get $\det(Aad_A(X)) = (2iau)^2$, which gives all squares. If $-1$ is not a square, then since $a\not\in {\mathbb F_q^*}^2$, $-a$ is a square. Let $-a =b^2$, for some $b \in \mathbb F_q^*$. Thus, $\det(Aad_A(X)) = a\big((r-av)^2 - 4au^2\big) = a\big((r-av)^2 + (2bu)^2\big)$. The sum $(r-av)^2 + (2bu)^2$ is an arbitrary sum of two squares in $\mathbb F_q$, and hence it covers all the values of $\mathbb F_q$. Since $a \ne 0$, $a\big((r-av)^2 + (2bu)^2\big)$ also takes all values over $\mathbb F_q$. Hence, all semisimple elements are in the image.   
\end{proof}

\section{Images of commutators of Engel words on $\Sl_2(q)$}\label{image3} 

In this section, we explicitly evaluate the images of certain Lie polynomials, which will be useful for obtaining images containing only split and anisotropic semisimple elements. We continue to work with $q$ odd. We start with Lie polynomials of the type $[ad_A^{i}(X), ad_A^{j}(X)]$, for $i\neq j \geq 1$ and later consider the Lie polynomials inductively defined as follows:
\begin{eqnarray*}
w_1 &=&\left[ ad_A^i(X)-ad_A^j(X), [ad_A^{i_1}(X), ad_A^{j_1}(X)]\right],\\
w_2 &=&\bigg[\left[ ad_A^i(X)-ad_A^j(X), [ad_A^{i_1}(X), ad_A^{j_1}(X)]\right], [ad_A^{i_2}(X), ad_A^{j_2}(X)]\bigg] \\ &&\quad \vdots \quad \quad \vdots \\
w_n &=& \Bigg[\cdots \bigg[ w_1, [ad_A^{i_2}(X), ad_A^{j_2}(X)]\bigg], \cdots, [ad_A^{i_n}(X), ad_A^{j_n}(X)]\Bigg] = \left[w_{n-1}, [ad_A^{i_n}(X), ad_A^{j_n}(X)]\right].
\end{eqnarray*}
We note that when $A=e+af$, we have
$A^n = \begin{pmatrix} a^{\frac{n}{2}} & 0 \\ 0 & a^{\frac{n}{2}} \end{pmatrix} = a^{\frac{n}{2}} I$ if $n$ is even, and $\begin{pmatrix} a^{\frac{n-1}{2}} & 0 \\ 0 & a^{\frac{n-1}{2}} \end{pmatrix}A = a^{\frac{n-1}{2}} A$ if $n$ is odd. 

\begin{proposition}\label{commutator of Engel words}
Let $w = [ad_A^{i}(X), ad_A^{j}(X)]$. Then the image of $w$ is,
$$w(\Sl_2(q) = \begin{cases}
    0 & \text{ if } i+j \text{ is even}\\
    \Sl_2(q) &  \text{ if } i+j=3\\
    \{0\} \cup \{\text{all semisimple elements}\} & \text{ otherwise.}
\end{cases}$$
\end{proposition}
\begin{proof}
By Remark~\ref{one variable over representatives}, to find the image of a Lie polynomial over $\Sl_2(q)$, we can vary $A$ over the representatives of automorphism orbits and then check which determinant values are possible. So, let $A = e + af$, for some $a \in \F_q$. By Lemma~\ref{higher ad maps}, $[ad_A^{i}(X), ad_A^{j}(X)] = [2^{i-1}A^{i-1}ad_A(X), 2^{j-1}A^{j-1}ad_A(X)] = 2^{i+j-2}[A^{i-1}ad_A(X), A^{j-1}ad_A(X)]$. We can put the values of powers of $A$. Thus, if $i$ and $j$ are simultaneously even or odd, then $[ad_A^{i}(X), ad_A^{j}(X)] = 0$. So, let $i+j$ be odd. Without loss of generality, let $i$ be even, and $j$ be odd (determinant values are unchanged even if we change the order of $i$ and $j$). We let $X = b_1e + b_2 f+ b_3h$. Then for $A = e + af$, $ad_A(X) = [e + af, b_1e + b_2 f+ b_3h] = -2b_3e + 2ab_3f+ (b_2-ab_1)h$. Thus,
\begin{align*}
& [ad_A^{i}(X), ad_A^{j}(X)]= 2^{i+j-2} a^{\frac{i-2}{2}} a^{\frac{j-1}{2}} [A ad_A(X), ad_A(X)] = 2^{i+j-2} a^{\frac{i+j-3}{2}}[A, ad_A(X)] ad_A(X)\\
&= 2^{i+j-2} a^{\frac{i+j-3}{2}} [e+af,\ -2b_3e + 2ab_3f+ (b_2-ab_1)h] \big(-2b_3e + 2ab_3f+ (b_2-ab_1)h\big).
\end{align*}
After solving the right-hand side of the above equation, we get
\begin{eqnarray*}
[ad_A^{i}(X), ad_A^{j}(X)] &=& 2^{i+j-1} a^{\frac{i+j-3}{2}} \big((ab_1-b_2)^2-4ab_3^2 \big) \begin{pmatrix} 0 & 1 \\ a & 0 \end{pmatrix}\\ & =& 2^{i+j-1} a^{\frac{i+j-3}{2}} \big((ab_1-b_2)^2-4ab_3^2 \big)A.
\end{eqnarray*}
The set of determinant values attained by elements in the image of $[ad_A^{i}(X), ad_A^{j}(X)]$ is 
$\{-4^{i+j-1} a^{i+j-2} \big((ab_1-b_2)^2-4ab_3^2 \big)^2 \mid b_1,b_2,b_3\in \F_q\}$. Even for a fixed $a \in \F_q^*$, this set is equal to $\F_q$. Thus, when $i+j$ is odd, all semisimple elements are present in the image of $w$. Also, the determinant value is zero with $[ad_A^{i}(X), ad_A^{j}(X)] \neq 0$, only when $i+j=3$. Hence, $w(\Sl_2(q))$ contains nilpotents if only if $i+j=3$.
\end{proof}

\begin{proposition}\label{prop: w_n image}
Let us consider the Lie polynomial $w_n$ where $i_r+j_r$ is odd for each $r \in \{1,2, \dots, n\}$, $\sigma = \sum_{r=1}^{n}(i_r+j_r) \geq 3n$, $i-j$ is positive and even. Then $w_n(\Sl_2(q))$ contains only semisimple elements with determinant equal to 
$$\varepsilon \ 4^{j-1+\sigma}  a^{j-2n-1+\sigma}(1-2^{i-j}a^{\frac{i-j}{2}})^2 b^{2n+1}$$
where $a,b \in \F_q$,  $\varepsilon = (-1)^{n}$ if $i$ and $j$ both are even; and $\varepsilon = (-1)^{n+1}$ if $i$ and $j$ both are odd.
\end{proposition}
\begin{proof}
We let $A = e +af$, and $X = b_1e +b_2f +b_3h$, for some $a, b_1, b_2, b_3 \in \mathbb F_q$. Without loss of generality, we let $i_r$ be even, and $j_r$ be odd. By Lemma~\ref{higher ad maps} and Proposition~\ref{commutator of Engel words}, we have
\begin{align*}
&w_n(A,X)\\
&= \Big[\cdots\big[2^{i-1}A^{i-1} ad_A(X)-2^{j-1} A^{j-1} ad_A(X),\ 2^{i_1+j_1-1} a^{\frac{i_1+j_1-3}{2}} \big((ab_1-b_2)^2-4ab_3^2 \big)A\big], \dots \\
& \dots,2^{i_n+j_n-1} a^{\frac{i_n+j_n-3}{2}} \big((ab_1-b_2)^2-4ab_3^2 \big)A \Big]\\
&= 2^{j+\sum_r (i_r+j_r) -n -1} a^{\sum_r \frac{i_r+j_r-3}{2}}\big((ab_1-b_2)^2-4ab_3^2 \big)^n\Big[\dots\big[[2^{i-j}A^{i-1}ad_A(X)  - A^{j-1}ad_A(X), \\& \ A],\ A\big], \dots, A\Big]\\
&= \kappa a^{\sum_r \frac{i_r+j_r-3}{2}}\big((ab_1-b_2)^2-4ab_3^2 \big)^n\Big[\dots\big[[2^{i-j}A^{i-1}ad_A(X)  - A^{j-1}ad_A(X), \ A],\ A\big], \dots, A\Big]
\end{align*}
where $\kappa=2^{j+\sigma -n -1}$, and $\sigma=\sum_r (i_r+j_r)$.

{\bf Case 1:} Let $i$ and $j$ be both even. Then, for $A = e+a f$, $A^{i-1}= a^{\frac{i-2}{2}}A$ and $A^{j-1}= a^{\frac{j-2}{2}}A$. Substituting it in the last equation, we get
\begin{align*}
&w_n(A,X)\\
&= \kappa a^{\frac{j-2 + \sum_r(i_r+j_r-3)}{2}} \big((ab_1-b_2)^2-4ab_3^2 \big)^n [\dots[[A\big(2^{i-j}a^{\frac{i-j}{2}}-1\big) ad_A(X),\  A],\ A], \dots, A]\\
&= \kappa \kappa_1 a^{\frac{j-2 + \sum_r(i_r+j_r-3)}{2}}\big((ab_1-b_2)^2-4ab_3^2 \big)^n[\dots[[A ad_A(X),\  A],\ A], \dots, A]
\end{align*}
where $\kappa_1=\big(2^{i-j}a^{\frac{i-j}{2}}-1\big)$. 
Note that
\begin{align*}
[\dots[[A ad_A(X),\   A],\ A], \dots, A]
&=A[\dots[[ad_A(X),\   A],\ A], \dots, A]\\
&= A (-1)^n ad_A^{n+1}(X) =(-1)^n2^{n}A^{n+1}ad_A(X).
\end{align*}
Thus, 
\begin{align*}
w_n(A,X) &=\kappa \kappa_1 2^n a^{\frac{j-2 + \sum_r(i_r+j_r-3)}{2}} \big((ab_1-b_2)^2-4ab_3^2 \big)^n (-1)^n2^{n}A^{n+1}ad_A(X)\\
&=(-1)^n\kappa_2\kappa_1 a^{\frac{j-2 + \sigma -3n}{2}}\big((ab_1-b_2)^2-4ab_3^2 \big)^n A^{n+1} \begin{pmatrix}
    b_2-ab_1 &-2b_3 \\ 2ab_3 &-b_2+ab_1
\end{pmatrix}\end{align*}
where $\kappa_2=\kappa 2^n$. This gives us,
\begin{align*}
\det(w_n(A,X)) &=\kappa_2^2\kappa_1^2 a^{j-2 + \sigma -3n} \big((ab_1-b_2)^2-4ab_3^2 \big)^{2n} (-a)^{n+1}(-(ab_1-b_2)^2+4ab_3^2)\\
&=(-1)^{n}\kappa_2^2\kappa_1^2 a^{j -2n-1 + \sigma } \big((ab_1-b_2)^2-4ab_3^2 \big)^{2n+1}.\end{align*}

{\bf Case 2:} If $i$ and $j$ are both odd, then for $A = e+a f$, $A^{i-1}= a^{\frac{i-1}{2}}I$ and $A^{j-1}= a^{\frac{j-1}{2}}I$. In this case, we get
\begin{align*}
w_n(A,X) & = \kappa\kappa_1 a^{\frac{j-1 + \sum_r(i_r+j_r-3)}{2}} \big((ab_1-b_2)^2-4ab_3^2 \big)^n[\dots[[ ad_A(X),\   A],\ A], \dots, A]\\
&= \kappa\kappa_1 a^{\frac{j-1 + \sum_r(i_r+j_r-3)}{2}} \big((ab_1-b_2)^2-4ab_3^2 \big)^n \big((-1)^n 2^n ad_A^{n+1}(X)\big)\\
&=(-1)^n\kappa_2\kappa_1 a^{\frac{j-1 + \sigma -3n}{2}} \big((ab_1-b_2)^2-4ab_3^2 \big)^n A^{n}ad_A(X)
\end{align*}
Hence, 
\begin{align*}
\det(w_n(A, X))& = \kappa_2^2\kappa_1^2 a^{j-1 + \sigma -3n} \big((ab_1-b_2)^2-4ab_3^2 \big)^{2n}(-a)^n\big(-(ab_1-b_2)^2+4ab_3^2 \big)\\
&=(-1)^{n+1}\kappa_2^2\kappa_1^2 a^{j-2n-1 + \sigma} \big((ab_1-b_2)^2-4ab_3^2 \big)^{2n+1}\end{align*} 

In both of the above cases, since $\sigma = \sum_{r=1}^{n}(i_r+j_r) \geq 3n$, $\det(w_n(A, X)) = 0$ if and only if $w_n(A, X)=0$, and hence $w_n(\Sl_2(q)$ contains no nilpotent elements. We have, for each $r$, the sum $i_r+j_r$ is odd and hence $\geq 3$. So, for $j>0$, $j - 2n -1 + \sigma > 0$. Therefore, if $a=0$, the expression $\varepsilon \kappa^2 \kappa_1^2 4^{j-1+\sum_{r}(i_r+j_r)} a^{j-2n-1+\sum_{r}(i_r+j_r)}(2^{i-j}a^{\frac{i-j}{2}}-1)^2 \big((ab_1-b_2)^2-4ab_3^2 \big)^{2n+1}$ is equal to $0$. Further, we claim that for any $a \in \mathbb{F}_q^*$, the polynomial $(ax_1 - x_2)^2 - 4ax_3^2$ is surjective on $\mathbb F_q$. For this, we first assume that $-a = c^2$ for some $c \in \mathbb F_q$. In this case, $(ax_1 - x_2)^2 - 4ax_3^2 = (ax_1-x_2)^2 + (2cx_3)^2$, which takes all values in $\mathbb F_q$ that can be written as a sum of two squares and hence it is surjective on $\mathbb F_q$. We now assume that $-a$ is not a square in $\mathbb F_q$. For $x_1 =0$, the polynomial $(ax_1-x_2)^2-4ax_3^2$ becomes $x_2^2 + (-a)(2x_3)^2 $ which takes all square values in $\mathbb F_q$ at $x_3 =0$ and all non-squares values in $\mathbb F_q$ at $x_2 =0$. Thus, for any $a \in {\mathbb{F}_q}^*$, the polynomial $(ax_1 - x_2)^2-4ax_3^2$ is surjective on $\mathbb F_q$. Therefore, in the expression of $\det(A, X)$, we can replace $(ab_1 - b_2)^2-4ab_3^2$ by $b$ and vary it over $\mathbb F_q$. This completes the proof.
\end{proof}

The polynomial $f(x, y)= \pm \ 4^{j-1 + \sum_{r}(i_r+j_r)} x^{j-2n-1+\sum_{r}(i_r+j_r)}(1-2^{i-j}x^{\frac{i-j}{2}})^2 y^{2n+1}$ need not be always surjective on $\mathbb F_q$. Thus, for a suitable choice of parameters $i, j, i_1, j_1\dots, i_r, j_r$ and $n$, the Lie polynomial $w_n$ defined earlier could have an image that does not contain all split or anisotropic semisimple elements.

The following is an easy computation that we include for completeness. 
\begin{lemma}
Let $\mathbb{F}_q$ be a finite field of odd characteristic and let $S = \{x^2 \mid x \in \mathbb{F}_q\}$ be the set of quadratic residues in $\mathbb{F}_q$. For any non-zero element $a \in \mathbb{F}_q$, the number of ordered pairs $(u, v) \in S \times S$ such that $u + v = a$, denoted by $N_{S}(a)$, is given by:
$$
N_{S}(a) = \frac{q - \eta(-1) + 2\eta(a) + 2}{4}
$$
where $\eta$ is the quadratic character (Legendre symbol) on $\mathbb{F}_q$.
\end{lemma}\label{N_s(a) count}
\begin{proof}
The number of solutions for the equation $x^2 + y^2 = a$ in $\mathbb{F}_q^2$ is $N(a) = q - \eta(-1)$. We partition these solutions as follows:

\textbf{Case 1.} Boundary Pairs (where $u=0$ or $v=0$):
The pair $(0, a)$ is a solution in $S \times S$ if and only if $a \in S$ (i.e., $\eta(a) = 1$). 
\begin{itemize}
    \item If $\eta(a) = 1$, then $(0, a)$ and $(a, 0)$ are two square pairs which correspond to four solutions $(0, \pm a)$ and $(\pm a, 0)$.
    \item If $\eta(a) = -1$, no such square pairs exist.
\end{itemize}
Thus, there are $1 + \eta(a)$ boundary square pairs that corresponds to $2(1 + \eta(a))$ solutions.

\textbf{Case 2.} Internal Pairs (where $u, v \neq 0$):
The number of solutions $(x, y)$ with $x, y \neq 0$ is:
\[
N_{nz} = N(a) - 2(1 + \eta(a)) = q - \eta(-1) - 2 - 2\eta(a)
\]
Each internal square pair $(u, v)$ corresponds to exactly four solutions $(\pm\sqrt{u}, \pm\sqrt{v})$. Therefore:
\[
\text{Internal Pairs} = \frac{q - \eta(-1) - 2 - 2\eta(a)}{4}
\]

Thus, the number of distinct pairs is
$$N_S(a) = (1 + \eta(a)) + \frac{q - \eta(-1) - 2 - 2\eta(a)}{4} = \frac{q - \eta(-1) + 2\eta(a) + 2}{4}.$$
\end{proof}

\begin{lemma}\label{lem: s_alpha values for x^2(1-4x)^2}
Let $f\colon \mathbb F_q \rightarrow \mathbb F_q$ be a map defined as $f(x) = x^2(1-4x)^2$. For any $\alpha$ a non-negative integer, let $s_\alpha = |\{a \in \mathbb F_q \mid |f^{-1}(a)|= \alpha\}|$. Then $s_\alpha =0$, for all $\alpha > 4$, and the values of $s_\alpha$ for $\alpha \in \{0,1,2,3,4\}$ is given in the following table.
\begin{table}[h]
\parbox{.7\linewidth}{
\centering
 \begin{tabular}{|c|c|c|c|c|c|}
 \hline
 $\alpha$ &$0$&$1$&$2$&$3$&$4$\\ \hline
 $s_\alpha$&$\frac{q+2\eta(2)-\eta(-1)-2}{8}$ &$\frac{1-\eta(2)}{2}$& $\frac{q+2+\eta(-1)}{4}$& $\frac{1+\eta(2)}{2}$&$\frac{ q - \eta(-1) - 2\eta(2) -6}{8}$ \\
 \hline
  \end{tabular}
\caption{$s_\alpha$ values}
    \label{tab:polynomial values for x^2(1-4x)^2}
}
\end{table}
\end{lemma}
\begin{proof}
Let $g(x) = x (1-4x)$ and $f(x)=g(x)^2$. Note that $g(x)=g(-x+4^{-1})$. Hence, $g(x)$ is a two-to-one function at all points, except for $x = -x + 4^{-1}$, i.e., $x = 8^{-1}$. Thus, $f(x)$ has at least two preimages for each $x \neq 8^{-1}$. Also, if $a \in \mathbb F_q \backslash\{0,4^{-1}\}$, then $g(a)$ and $-g(a)$ are two distinct preimages for $f(a)$. By these properties of $f$ and $g$, we conclude that if $a \in \mathbb F_q \backslash\{0,4^{-1}\}$ and $f(a) \ne f(8^{-1})$, then $f(a)$ has $4$ preimages if $-g(a) \in \img(g)$, and $f(a)$ has $2$ preimages if $-g(a) \notin \img(g)$. Clearly, $f(x) = 0$ at $x =0$ and $x =4^{-1}$. So, $f(0)$ has precisely two preimages. If $-g(8^{-1}) = g(y)$, for some $y$, then there are three preimages for $f(8^{-1})$; namely $8^{-1}$, $y$ and $-y + 4^{-1}$. If $-g(8^{-1}) \ne g(y)$, for all $y$, then there is a unique preimage for $f(8^{-1})$.

Now, let us check when $-g(8^{-1})$ lies in the image of $g(x)$. If so happens, then there exist $y \in \mathbb F_q$ such that $-g(8^{-1})= g(y) \implies -16^{-1} = y (1-4y) \implies 4y^2 -y -16^{-1}=0$. The last relation holds if and only the discriminant of the quadratic polynomial  $4y^2 -y -16^{-1}$ is a square in $\mathbb F_q$. Hence, such a $y$ exists if and only if $1-4(4)(-16^{-1}) = 2$ is a square in $\mathbb F_q$. So, $-g(8^{-1})$ lies in the image of $g(x)$ if $q \equiv1, 7 \mod 8$ and $-g(8^{-1})$ does not lie in the image of $g(x)$ if $q \equiv 3,5 \mod 8$. Therefore, if $q \equiv 1, 7 \mod 8$, then $s_1 = 0 = \frac{1-\eta(2)}{2}$, $s_3 = 1 = \frac{1+\eta(2)}{2}$ and $s_\alpha=0$, for all $\alpha > 4$, and if $q \equiv 3, 5 \mod 8$, then $s_1 = 1 = \frac{1-\eta(2)}{2}$, $s_3 = 0 = \frac{1+\eta(2)}{2}$ and $s_\alpha =0$, for all $\alpha > 4$.

We now proceed to find $s_4$. As discussed above, $|f^{-1}(a)| = 4$ for some $a \in \mathbb F_q \backslash\{0,4^{-1}\}$ if and only if $f(a) \neq f(8^{-1})$ and $-g(a) \in \img(g)$. We aim to count such $a$ in $\mathbb F_q$. Let $g(a) = k$. If $-k \in \img(g)$, then there exists a solution in $\mathbb F_q$ for the equation $x-4x^2=-k$. Thus, the discriminant of the polynomial $4x^2-x-k$ is a square. In other words, $1-16 k$ is a square in $\mathbb F_q$. Since $g(a)=k$, $k$ also satisfies the equation $x-4x^2=k$ and hence $1+16 k$ is a square in $\mathbb F_q$. Therefore, $k$ and $-k$ are both in $\img(g)$ if and only if $1+16 k$ and $1-16k$ are squares in $\mathbb F_q$. This is equivalent to saying that $(1+16 k) + (1-16 k) =2$ is a sum of two squares. Thus, counting $k \in \mathbb F_q$ such that $k$ and $-k$ are both in $\img(g)$ is equivalent to counting the number of distinct solutions for the equation $u+v =2$, $u,v \in \mathbb F_q$. By Lemma \ref{N_s(a) count}, this count is $N_{S}(2) = \frac{q - \eta(-1) + 2\eta(2) + 2}{4}$. In addition to it, we want $k$ to be non-zero, since $k=0$ corresponds to $f(a)=0$, which has exactly two preimages. Further, the pairs $(k,-k)$ and $(-k, k)$ accounts for double counting if $k \ne 0$. Thus, the total number of ordered non-zero pairs $(k,-k)$ such that both $k$ and $-k$ are in the image are $\frac{N_{S}(2)-1}{2} = \frac{ q - \eta(-1) + 2\eta(2) -2}{8}$. Also, one of these pair is $(8^{-1}, -8^{-1})$ which corresponds for $s_3$ not $s_4$. This happens if and only if $\eta(2) =1$. Thus, $s_4 = \frac{ q - \eta(-1) + 2\eta(2) -2}{8} - \frac{1+\eta(2)}{2} = \frac{ q - \eta(-1) - 2\eta(2) -6}{8}$.

Since $s_\alpha =0$, for all $\alpha> 4$, we have the relation $s_1 + 2s_2 +3 s_3 +4s_4 = q$. By substituting the values of $s_1$, $s_3$ and $s_4$, we get $s_2 = \frac{q+ 2+ \eta(-1)}{4}$. Finally, $s_0$ is the number of squares which are not in the image of $f$, and therefore it is equal to $\frac{q+1}{2} - s_1 - s_2 - s_3 - s_4 = \frac{q + 2\eta(2) -\eta(-1)-2}{8}$.
\end{proof}

The following theorem guarantees the existence of Lie polynomials whose images consist entirely of semisimple elements, though not all such elements may be contained in the image.

\begin{theorem}\label{thm:Semisimple elements of same type missing}
Consider the Lie polynomial 
$$w_n = \Bigg[\dots\Big[ad_A^i(X)-ad_A^j(X), \ [ad_A^{i_1}(X), ad_A^{j_1}(X)]\Big],\dots,\ [ad_A^{i_n}(X), ad_A^{j_n}(X)]\Bigg]$$ where $i_r+j_r$ is odd for each $r \in \{1,2, \dots, n\}$, $\sigma=\sum_{r=1}^{n}(i_r+j_r) \geq 3n$, $i-j$ is positive and even. Then $w_n(\Sl_2(q))$ contains only semisimple elements. Moreover,
\begin{enumerate}
\item Let $q \equiv 3 \mod 8$. If $n=\frac{q-3}{4}$, $j$ is odd, $i=j+2$, and $i_1,j_1, \dots, i_n,j_n$ be such that $j-2n-1+\sigma \equiv 2 \mod (q-1)$, then $w_n(\Sl_2(q))$ misses exactly $\frac{q-3}{4}$ semisimple orbits, half of which are split and half of which are anisotropic.
\item Let $q \equiv 7 \mod 8$. If $n=\frac{q-3}{4}$, $j$ is even, $i=j+2$, and $i_1,j_1, \dots, i_n,j_n$ be such that $j-2n-1+\sigma \equiv 2 \mod (q-1)$, then $w_n(\Sl_2(q))$ misses exactly $\frac{q+1}{4}$ semisimple orbits, half of which are split and half of which are anisotropic.
\end{enumerate}
\end{theorem}
\begin{proof}
By Proposition \ref{prop: w_n image}, $w_n(\Sl_2(q))$ contains only semisimple elements with determinant equal to 
$$\varepsilon \ 4^{j-1+\sum_{r}(i_r+j_r)} a^{j-2n-1+\sum_{r}(i_r+j_r)}(1-2^{i-j}a^{\frac{i-j}{2}})^2 b^{2n+1}$$
where $a,b \in \F_q$,  $\varepsilon = (-1)^{n}$; if $i$ and $j$ are even and $\varepsilon = (-1)^{n+1}$; if $i$ and $j$ are odd. 

If $q \equiv 3 \mod 4$, then $\frac{q-3}{4}$ is an integer. So, we can take $n=\frac{q-3}{4}$ which in turn implies that $2n+1 = \frac{q-1}{2}$. If $b \ne 0$, then $b^{\frac{q-1}{2}}$ is either $1$ or $-1$. Thus, as $b$ varies over $\mathbb F_q$, $b^{\frac{q-1}{2}}$ takes values $0$, $1$ and $-1$. Further, if $q \equiv 3 \mod 8$, then $n=\frac{q-3}{4}$ is even and hence the sum $\sum_{r=1}^{n}(i_r+j_r)$ is always even because $i_r+j_r$ is odd for all $r$. Thus, there exists $i_1,j_1, \dots, i_n,j_n$ such that $j-2n-1+\sum_{r}(i_r+j_r) \equiv 2 \mod q-1$ if and only if $j$ is odd. If $q \equiv 7 \mod 8$, then $n=\frac{q-3}{4}$ is odd and hence the sum $\sum_{r=1}^{n}(i_r+j_r)$ is always odd. Thus, there exists $i_1,j_1, \dots, i_n,j_n$ such that $j-2n-1+\sum_{r}(i_r+j_r) \equiv 2 \mod q-1$ if and only if $j$ is even. If we substitute $n=\frac{q-3}{4}$, $i-j=2$, and $i_1,j_1, \dots, i_n,j_n$ be such that $j-2n-1+\sum_{r}(i_r+j_r) \equiv 2 \mod q-1$, then the possible determinant values are $\varepsilon \ 4^{\frac{q+1}{2}} a^{2}(1-4a)^2 b^{\frac{q-1}{2}}$. Hence, in this case 
$$w_n(\Sl_2(q)) = \{P \in \Sl_2(q) \mid \det(P)  \ne 0 \text{ and } \det(P)= 4^{\frac{q+1}{2}}c \text{ or } -4^{\frac{q+1}{2}}c,$$ 
where $c\in \img(x^{2}(1-4x)^2)\} \cup \{0\}$. 

Since $-1$ is not a square in $\mathbb F_q$ when $q \equiv 3,7 \mod 8$, $w_n(\Sl_2(q))$ misses the same number of split and anisotropic semisimple orbits. By Lemma \ref{lem: s_alpha values for x^2(1-4x)^2}, the polynomial $x^{2}(1-4x)^2)$ misses $\frac{q-3}{8}$ non-zero square values when $q \equiv 3 \mod 8$, and $\frac{q+1}{8}$ non-zero square values when $q \equiv 7 \mod 8$. Thus, $w_n(\Sl_2(q))$ misses exactly $\frac{q-3}{4}$ semisimple orbits when $q \equiv 3 \mod 8$, and $\frac{q+1}{4}$ semisimple orbits when $q \equiv 7 \mod 8$, half of which are split and half of which are anisotropic. 
\end{proof}

The above theorem gives Lie polynomials whose images over $\Sl_2(q)$, $q \equiv 3 \mod 4$, consist of semisimple elements, but not all semisimple elements are contained in the image. However, the general idea applies to any $q$ for which $q-1$ has a sufficiently large odd divisor, say $d_{q-1}$. For such $q$, we take $n = \frac{d_{q-1}-1}{2}$ and $j$ such that $n+j$ is odd, $i = j+2$, and $i_1,j_1, \dots, i_n,j_n$ be such that each $i_r+j_r$ is odd and $j-2n-1+\sum_{r}(i_r+j_r) \equiv 2 \mod q-1$. With these assumptions, we can obtain Lie polynomials $w_n$ such that  $w_n(\Sl_2(q))$ does not contain all semisimple elements, even for fields $\F_q$, satisfying $q \not\equiv 3 \mod 4$. 
\begin{example}
For example, if $q= 29$, we take $d_{q-1}=7$, $n=3$, $j=2$, $i=4$, and $i_1, j_1, i_2, j_2, i_3, j_3$ be such that each $i_r+j_r$ is odd and $j-2n-1+\sum_{r}(i_r+j_r) = -5 +\sum_{r}(i_r+j_r)\equiv 2 \mod q-1$. With these assumptions, the word $w_n$ has its image over $\Sl_2(q)$ with the possible determinant values to be $\varepsilon \ 4^{8} a^{2}(1-4a)^2 b^{7}$. We use \textsf{GAP} \cite{GAP4.14.0} to check that the polynomial $x^{2}(1-4x)^2 y^{7}$ misses exactly $4$ values over $\mathbb F_{29}$. Hence, $w_n\big(\Sl_2(q)\big)$ misses exactly $4$ semisimple orbits together with the nilpotent.
\end{example}

\section{Images with no nilpotents and few semisimple orbits on $\Sl_2(q)$}\label{image4}

In this section, we aim to obtain Lie polynomials that give a single conjugacy class, with $0$ as its image. We consider the following Lie polynomials:
$$w_{1,n}:= \bigg[\dots\Big[ad_A^{\alpha_1}(X)-ad_A^{\beta_1}(X), \ [ad_A^{i_1}(X), ad_A^{j_1}(X)]\Big],\dots,\ [ad_A^{i_n}(X), ad_A^{j_n}(X)]\bigg]$$ 
and inductively define, for $m\geq 2$,
$$w_{m,n}:= [w_{m-1,n}, ad_A^{\alpha_{m}}(X) - ad_A^{\beta_{m}}(X)].$$ 
Note that the Lie polynomial $w_{1,n}=w_n$ was dealt with in the previous section.

\begin{lemma}\label{prop: few semisimple orbits w_m,n}
Consider the Lie polynomial $w_{m,n}$ for $m\geq 2$ with (i) $i_r+j_r$ is odd for all $r$, (ii) $\alpha_m-\beta_m$ is positive for all $m$ and even, (iii) $\beta_1 + \beta_2 +n$ is odd, and (iv) $\sum_s\beta_s + \sigma -2n-m-1>0$ with $\sigma=\sum_r(i_r+j_r)$. 
Then, we have the following:
\begin{enumerate}
\item If $\beta_3,\beta_4,\dots, \beta_{2m_1}$ are even for some $m_1 \leq [\frac{m}{2}]$, and $\beta_{2m_1 +1}, \beta_{2m_1 + 2}, \dots \beta_m$ are odd, then $w_{m,n}(\Sl_2(q))$ contains only semisimple elements with determinant equal to 
$$(-1)^{m+1} 4^{\sum_s\beta_s+\sigma -1} a^{\sum\beta_s + \sigma-2n-m}\big(2^{\alpha_1-\beta_1}a^{\frac{\alpha_1-\beta_1}{2}}-1\big)^2\cdots \big(2^{\alpha_m-\beta_m}a^{\frac{\alpha_m-\beta_m}{2}}-1\big)^2 b^{2n+m}$$
where $a,b \in \F_q$.
\item If $\beta_3,\beta_4,\dots, \beta_{m-1}$ and $m$ are odd, and $\beta_{m}$ is even, then $w_{m,n}(\Sl_2(q))$ contains only semisimple elements with determinant equal to 
$$(-4^{\sum_s\beta_s+\sigma -1}a^{\sum\beta_s + \sigma -2n-m}\big(2^{\alpha_1-\beta_1}a^{\frac{\alpha_1-\beta_1}{2}}-1\big)^2\cdots \big(2^{\alpha_m-\beta_m} a^{\frac{\alpha_m-\beta_m}{2}}-1\big)^2 b^{2n+m}$$
where $a,b \in \F_q$.
\end{enumerate}
\end{lemma}
\begin{proof}
We let $A = e +af$, and $X = b_1e+b_2f+b_3h$, for some $a, b_1, b_2, b_3 \in \mathbb F_q$. Without loss of generality, we let $i_r$ be even, and $j_r$ be odd. By definition
$$w_{m,n}= \bigg[\dots\Big[[w_n, ad_A^{\alpha_{2}}(X) - ad_A^{\beta_{2}}(X)], \ ad_A^{\alpha_{3}}(X) - ad_A^{\beta_{3}}(X)\Big], \cdots, \ ad_A^{\alpha_{m}}(X) - ad_A^{\beta_{m}}(X)\bigg].$$
We split the proof into two cases.

\textbf{Case 1:} $\beta_1$ is even.
By Lemma \ref{higher ad maps} and Proposition \ref{prop: w_n image},
\begin{align*}
&w_{m,n}(A,X) = \\
&(-1)^n 2^{\beta_1+\sum_r (i_r+j_r) -1}a^{\frac{\beta_1-2 + \sum_r(i_r+j_r)-3n}{2}}\big(2^{\alpha_1-\beta_1}a^{\frac{\alpha_1-\beta_1}{2}}-1\big)\big((ab_1-b_2)^2-4ab_3^2 \big)^n[\dots[[A^{n+1} ad_A(X),\\
&2^{\alpha_2-1}A^{\alpha_2-1} ad_A(X)- 2^{\beta_2-1}A^{\beta_2-1} ad_A(X)],\cdots, 2^{\alpha_m-1}A^{\alpha_m -1} ad_A(X)- 2^{\beta_m-1}A^{\beta_m-1} ad_A(X)].  
\end{align*}
Note that $2^{\alpha_s-1}A^{\alpha_s -1} ad_A(X)- 2^{\beta_s-1}A^{\beta_s-1} ad_A(X)= 2^{\beta_s-1}A^{\beta_s-1}(2^{\alpha_s-\beta_s}A^{\alpha_s-\beta_s}-1)ad_A(X) = (2^{\alpha_s-\beta_s}a^\frac{\alpha_s-\beta_s}{2}-1)2^{\beta_s-1}A^{\beta_s-1}ad_A(X)$. Thus,
\begin{align*}
&w_{m,n}(A,X) = \\
&(-1)^n 2^{\sum_s(\beta_s-1)+\sum_r (i_r+j_r)}a^{\frac{\beta_1-2 + \sum_r(i_r+j_r)-3n}{2}}\big(2^{\alpha_1-\beta_1}a^{\frac{\alpha_1-\beta_1}{2}}-1\big)\dots \big(2^{\alpha_m-\beta_m}a^{\frac{\alpha_m-\beta_m}{2}}-1\big)\\ & \big((ab_1-b_2)^2-4ab_3^2 \big)^n
[\dots[[A^{n+1} ad_A(X),\ A^{\beta_2-1} ad_A(X)],\cdots, A^{\beta_m-1} ad_A(X)].  
\end{align*}

We now compute the $w':=[\dots[[A^{n+1} ad_A(X),\ A^{\beta_2-1} ad_A(X)],\cdots, A^{\beta_m-1} ad_A(X)]$ separately. For simplicity, we let $B:= ad_A(X)$. By our assumption $\beta_1 + \beta_2 + n$ is odd and $\beta_1$ is even. Hence, $\beta_2 +n$ is odd. We first assume that $n$ is odd, then $\beta_2$ is even. Thus, $[A^{n+1} B,\ A^{\beta_2-1} B] = a^{\frac{n+1+\beta_2-2}{2}}[B, A B] =a^{\frac{n+\beta_2-1}{2}}[B, A]B=-a^{\frac{n+\beta_2-1}{2}}ad_A^2(X)B=-2a^{\frac{n+\beta_2-1}{2}}AB^2$. Evidently, $w' = -2a^{\frac{n+\beta_2-1}{2}}[\dots[[AB^2,\ A^{\beta_3-1} B],\cdots, A^{\beta_m-1} B]$. If $n$ is even, then $\beta_2$ is odd. Thus, $[A^{n+1} B,\ A^{\beta_2-1} B] = a^{\frac{n+\beta_2-1}{2}}[AB, B]=2a^{\frac{n+\beta_2-1}{2}}AB^2$. Hence, $w' = 2a^{\frac{n+\beta_2-1}{2}}[\dots[[AB^2,\ A^{\beta_3-1} B],\cdots, A^{\beta_m-1} B]$. Combining the two scenarios, we write $w' = (-1)^{\beta_2}2a^{\frac{n+\beta_2-1}{2}}[\dots[[AB^2,\ A^{\beta_3-1} B],\cdots, A^{\beta_m-1} B]$.\\

We first assume that $\beta_3, \beta_4, \dots, \beta_{2m_1}$ are even and $\beta_{2m_1+1},\dots, \beta_{m}$ are odd, therefore $\beta_3-1, \beta_4-1, \dots, \beta_{2m_1}-1$ are odd, and $\beta_{2m_1+1}-1,\dots, \beta_{m}-1$ are even. Consequently,
$$w'= (-1)^{\beta_2}2a^{\frac{n+\beta_2-1+\beta_3-2+ \dots \beta_{2m_1}-2 + \beta_{2m_1+1}-1 \dots \beta_{m}-1}{2}}[\dots[[AB^2,\ \underbrace{A B],\cdots, AB}_{2m_1-2 \ \text{times}}], \underbrace{B], \cdots, B}_{m-2m_1 \ \text{times}}].$$
Since $ad_A^2(X) = [A, ad_A(X)] = 2Aad_A(X)$, we get $Aad_A(X) = -ad_A(X)A$, and hence $AB^j= (-1)^j B^j A$. By using this and the fact that $A^2$ is a scalar matrix, we have $[\dots[AB^2,\ \underbrace{A B],\cdots, AB}_{2m_1-2 \ \text{times}}] = (-1)^{m_1-1}2^{2m_1-2} A^{2m_1-1}B^{2m_1}= (-1)^{m_1-1}2^{2m_1-2} a^{m_1-1} AB^{2m_1}$. Thus, 
$$w'= (-1)^{\beta_2+m_1-1}2^{2m_1-1}a^{\frac{n+\sum_{s=2}^m\beta_s-1-2(2m_1-2)-(m -{2m_1})+2m_1-2}{2}}[\cdots[AB^{2m_1}, \underbrace{B], \cdots, B}_{m-2m_1 \ \text{times}}].$$
Now, by using $AB^j= (-1)^j B^j A$, we get $[\cdots [AB^{2m_1}, \underbrace{B], \cdots, B}_{m-2m_1 \ \text{times}}] = 2^{m-2m_1}AB^m$. Hence, 
$$w'= (-1)^{\beta_2+m_1-1}2^{m-1}a^{\frac{n+1-m+\sum_{s=2}^m\beta_s}{2}}AB^m.$$
Evidently,
\begin{align*}
&w_{m,n}(A,X) = (-1)^{m_1} 2^{\sum_s\beta_s+\sum_r (i_r+j_r)-1}a^{\frac{\sum\beta_s + \sum_r(i_r+j_r)-2n-m-1}{2}}\big(2^{\alpha_1-\beta_1}a^{\frac{\alpha_1-\beta_1}{2}}-1\big)\dots \\& \dots \big(2^{\alpha_m-\beta_m}a^{\frac{\alpha_m-\beta_m}{2}}-1\big)\big((ab_1-b_2)^2-4ab_3^2 \big)^n A (ad_A(X))^m.  
\end{align*}
Thus,
\begin{align*}
&\det(w_{m,n}(A,X)) = (-1)^{m+1}4^{\sum_s\beta_s+\sum_r (i_r+j_r)-1}a^{\sum\beta_s + \sum_r(i_r+j_r)-2n-m}\big(2^{\alpha_1-\beta_1}a^{\frac{\alpha_1-\beta_1}{2}}-1\big)^2\dots \\& \dots \big(2^{\alpha_m-\beta_m}a^{\frac{\alpha_m-\beta_m}{2}}-1\big)^2\big((ab_1-b_2)^2-4ab_3^2 \big)^{2n+m}.
\end{align*}

We now assume that $\beta_3,\beta_4,\dots, \beta_{m-1}$ and $m$ are odd, and $\beta_{m}$ is even. Then, 
\begin{align*}
 w' &= (-1)^{\beta_2}2a^{\frac{n+\beta_2-1}{2}}[\dots[[AB^2,\ A^{\beta_3-1} B],\cdots, A^{\beta_m-1} B]\\
 &= (-1)^{\beta_2}2a^{\frac{n+\sum_{s=2}^{m-1}(\beta_s-1)+ \beta_m-2}{2}}[\dots[[AB^2,\ B],\cdots, B],\ AB]\\
 &= (-1)^{\beta_2}2a^{\frac{n+\sum_{s=2}^{m-1}(\beta_s-1)+ \beta_m-2}{2}}\big( 2^{m-3}A^2B^{m-1} \big) = (-1)^{\beta_2}2^{m-1}a^{\frac{n-m+2+\sum_{s=2}^{m}\beta_s}{2}} B^{m} \\
\end{align*}
Evidently,
\begin{align*}
w_{m,n}(A,X) = &- 2^{\sum_s\beta_s+\sum_r (i_r+j_r)-1}a^{\frac{\sum\beta_s + \sum_r(i_r+j_r)-2n-m}{2}}\big(2^{\alpha_1-\beta_1}a^{\frac{\alpha_1-\beta_1}{2}}-1\big)\dots \\& \dots \big(2^{\alpha_m-\beta_m}a^{\frac{\alpha_m-\beta_m}{2}}-1\big)\big((ab_1-b_2)^2-4ab_3^2 \big)^n  (ad_A(X))^m.  
\end{align*}
Thus,
\begin{align*}
&\det(w_{m,n}(A,X)) = (-1)^{m}4^{\sum_s\beta_s+\sum_r (i_r+j_r)-1}a^{\sum\beta_s + \sum_r(i_r+j_r)-2n-m}\big(2^{\alpha_1-\beta_1}a^{\frac{\alpha_1-\beta_1}{2}}-1\big)^2\dots \\& \dots\big(2^{\alpha_m-\beta_m}a^{\frac{\alpha_m-\beta_m}{2}}-1\big)^2\big((ab_1-b_2)^2-4ab_3^2 \big)^{2n+m}.
\end{align*}

\textbf{Case 2:} $\beta_1$ is odd.\quad By Lemma \ref{higher ad maps} and Proposition \ref{prop: w_n image}, 
\begin{align*}
&w_{m,n}(A,X) = \\
&(-1)^n 2^{\beta_1+\sum_r (i_r+j_r) -1}a^{\frac{\beta_1-1 + \sum_r(i_r+j_r)-3n}{2}}\big(2^{\alpha_1-\beta_1}a^{\frac{\alpha_1-\beta_1}{2}}-1\big)\big((ab_1-b_2)^2-4ab_3^2 \big)^n[\dots[[A^{n} ad_A(X),\\
&2^{\alpha_2-1}A^{\alpha_2-1} ad_A(X)- 2^{\beta_2-1}A^{\beta_2-1} ad_A(X)],\cdots, 2^{\alpha_m-1}A^{\alpha_m -1} ad_A(X)- 2^{\beta_m-1}A^{\beta_m-1} ad_A(X)].  
\end{align*}
By using $2^{\alpha_s-1}A^{\alpha_s -1} ad_A(X)- 2^{\beta_s-1}A^{\beta_s-1} ad_A(X)=  (2^{\alpha_s-\beta_s}a^\frac{\alpha_s-\beta_s}{2}-1)2^{\beta_s-1}A^{\beta_s-1}ad_A(X)$, we get
\begin{align*}
&w_{m,n}(A,X) = \\
&(-1)^n 2^{\sum_s(\beta_s-1)+\sum_r (i_r+j_r)}a^{\frac{\beta_1-1 + \sum_r(i_r+j_r)-3n}{2}}\big(2^{\alpha_1-\beta_1}a^{\frac{\alpha_1-\beta_1}{2}}-1\big)\dots \big(2^{\alpha_m-\beta_m}a^{\frac{\alpha_m-\beta_m}{2}}-1\big)\\
&\big((ab_1-b_2)^2-4ab_3^2 \big)^n [\dots[[A^{n} ad_A(X),\ A^{\beta_2-1} ad_A(X)],\cdots, A^{\beta_m-1} ad_A(X)].  
\end{align*}
Let $w'':=[\dots[[A^{n} ad_A(X),\ A^{\beta_2-1} ad_A(X)],\cdots, A^{\beta_m-1} ad_A(X)]$. For simplicity, we again let $B:= ad_A(X)$. By our assumption $\beta_1 + \beta_2 + n$ and $\beta_1$ are odd. Hence, $\beta_2 +n$ is even. We proceed similarly as in the above case and get $w'' = (-1)^{\beta_2+1}2a^{\frac{n+\beta_2-2}{2}}[\dots[[AB^2,\ A^{\beta_3-1} B],\cdots, A^{\beta_m-1} B]$. Hence, 
$$w''= (-1)^{\beta_2+m_1}2^{m-1}a^{\frac{n-m+\sum_{s=2}^m\beta_s}{2}}AB^m.$$
Consequently,
\begin{align*}
&w_{m,n}(A,X) = (-1)^{m_1} 2^{\sum_s\beta_s+\sum_r (i_r+j_r)-1}a^{\frac{\sum\beta_s + \sum_r(i_r+j_r)-2n-m-1}{2}}\big(2^{\alpha_1-\beta_1}a^{\frac{\alpha_1-\beta_1}{2}}-1\big)\dots \\
& \dots \big(2^{\alpha_m-\beta_m}a^{\frac{\alpha_m-\beta_m}{2}}-1\big)\big((ab_1-b_2)^2-4ab_3^2 \big)^n A (ad_A(X))^m.  
\end{align*}
Hence,
\begin{align*}
&\det(w_{m,n}(A,X)) = (-1)^{m+1}4^{\sum_s\beta_s+\sum_r (i_r+j_r)-1}a^{\sum\beta_s + \sum_r(i_r+j_r)-2n-m}\big(2^{\alpha_1-\beta_1}a^{\frac{\alpha_1-\beta_1}{2}}-1\big)^2\dots \\
&\dots \big(2^{\alpha_m-\beta_m}a^{\frac{\alpha_m-\beta_m}{2}}-1\big)^2 \big((ab_1-b_2)^2-4ab_3^2 \big)^{2n+m}.
\end{align*}

We proceed similarly when $\beta_3,\beta_4,\dots, \beta_{m-1}$ and $m$ are odd, and $\beta_{m}$ is even. Further, in both the above cases, $\det(w_{m,n}(A,X))= 0$ if and only if $(w_{m,n}(A,X))=0$ and hence $w_{m,n}(\Sl_2(q))$ contains no nilpotent elements. Further, for any $a \in \mathbb{F}_q^*$, the polynomial $(ax_1-x_2)^2-4ax_3^2$ is surjective on $\mathbb F_q$. Therefore, in the expression of $\det(A,X)$, we can replace $(ab_1-b_2)^2-4ab_3^2$ by $b$ and vary it over $\mathbb F_q$. This completes the proof.
\end{proof}

\begin{remark}\label{remark:Few semisimple orbits}
The polynomial 
$$f(x,y):=x^{\sum\beta_s + \sum_r(i_r+j_r)-2n-m}\big(2^{\alpha_1-\beta_1}x^{\frac{\alpha_1-\beta_1}{2}}-1\big)^2\dots \big(2^{\alpha_m-\beta_m}x^{\frac{\alpha_m-\beta_m}{2}}-1\big)^2 y^{2n+m}$$ 
need not be surjective on $\mathbb F_q$ always. Thus, for a suitable choice of parameters, the Lie polynomial $w_{m,n}$ has an image that does not contain all split or anisotropic semisimple elements. For example, if $\alpha_s-\beta_s$ is same for all $s \in \{1,2,\dots, m\}$, then $f(x,y) = x^{\sum\beta_s + \sum_r(i_r+j_r)-2n-m}\big(2^{\alpha_1-\beta_1}x^{\frac{\alpha_1-\beta_1}{2}}-1\big)^{2m} y^{2n+m}$. Further, for $m= q-1$ and $n=\frac{q-1}{2}$, the set of non-zero values of the polynomial $f(x,y)$ is equal to the set of non-zero values of $x^{\sum_{s=1}^m\beta_s + \sum_{r=1}^n(i_r+j_r)}$ over $\F_q$. Since $n+\beta_1 + \beta_2$ is odd and $i_r + j_r$ is odd for each $r$, the sum $\sum_{s=1}^m\beta_s + \sum_{r=1}^n(i_r+j_r)$ is always odd. Moreover, if $c \in \{1,3, \dots, q-2\}$, then $c$ can be written as $\sum_{s=1}^m\beta_s + \sum_{r=1}^n(i_r+j_r)$ modulo $(q-1)$.
\end{remark}

\begin{proposition}\label{Prop: few semisimple orbits for odd gamma}
For $m\geq 2$ and $q$ odd, consider $w_{m,n}$, where $i_r+j_r$ is odd for all $r$; $\alpha_s-\beta_s =c$ where $c$ is either $2$ or $q-1$ for all $s$; $\beta_1 + \beta_2 +n$ is odd; $\beta_3, \beta_4, \dots, \beta_{2m_1}$ are even for some $m_1 \leq [\frac{m}{2}]$; $\beta_{2m_1 +1}, \beta_{2m_1 + 2}, \dots \beta_m$ are odd, and $\sum\beta_s + \sum_r(i_r+j_r)-2n-m-1>0$. Let $\gamma>1$ be any odd divisor of $q-1$. If $\sum_{s=1}^m\beta_s + \sum_{r=1}^n(i_r+j_r) \equiv \gamma \mod (q-1)$, then for any $m$ and $n$ satisfying $2m \equiv 0 \mod{(q-1)}$ and $2n + m \equiv 0 \mod{(q-1)}$, the set $w_{m,n}(\Sl_2(q))$ contains only semisimple elements with determinant equal to 
$$-4^{\gamma-1}a^{\gamma}$$
where $a \in \F_q^*$ if $c=2$, and $a\in {\mathbb F_q^*}^2$ if $c=q-1$. Evidently, if $c=2$, then $w_{m,n}(\Sl_2(q))$ contains only $\frac{q-1}{\gamma}$ semisimple orbits, half of them are split, and half are anisotropic; and if $c = q-1$, then $w_{m,n}(\Sl_2(q))$ contains only $\frac{q-1}{2\gamma}$ semisimple orbits, all of them are anisotropic.
\end{proposition}
\begin{proof}
Since $\alpha_s-\beta_s =c$ for all $s$, $\sum_{s=1}^m\beta_s + \sum_{r=1}^n(i_r + j_r) \equiv \gamma \mod (q-1)$, $2m \equiv 0 \mod{(q-1)}$ and $2n + m \equiv 0 \mod{(q-1)}$, the set $w_{m,n}(\Sl_2(q))$, by Lemma~\ref{prop: few semisimple orbits w_m,n}, contains only semisimple elements with determinant of the form $-4^{\gamma-1}a^{\gamma}\big(2^{c}a^{\frac{c}{2}}-1\big)^{2m}$, for some $a \in \F_q^*$.

We first assume that $c=2$. In this case, $2^{c}a^{\frac{c}{2}}-1$ is non-zero if and only if $a \neq 4^{-1}$. Hence, $w_{m,n}(\Sl_2(q))$ contains only semisimple elements with determinant of the form $-4^{\gamma-1}a^{\gamma}$, for some $a \in \F_q \backslash \{0, 4^{-1}\}$. The map $f_{\gamma} \colon \F_q^* \to \F_q^*$, sending $a$ to $a^{\gamma}$ is a $\gamma$-one map. Thus, for $q>3$ and $\gamma>1$, there exists $a' \in \F_q \backslash \{4^{-1}\}$ with $(a')^{\gamma} = {(4^{-1})}^{\gamma}$ and hence the number of semisimple orbits in $w_{m,n}(\Sl_2(q))$ is equal to the number of elements in $\F_q^*$ which can be written as $\gamma$-th powers. Also, the set of $\gamma$-powers in $\F_q^*$ forms a cyclic subgroup of order $\frac{q-1}{\gamma}$, which is an even integer if $q$ is odd. Hence, the number of squares in this subgroup is exactly half of $\frac{q-1}{\gamma}$. Consequently, the set $w_{m,n}(\Sl_2(q))$ contains only $\frac{q-1}{\gamma}$ semisimple orbits, half of them are split, and half are anisotropic.

Now, let $c=q-1$ and $a \in \F_q^*$. In this case, $2^{c}a^{\frac{c}{2}}-1 = a^{\frac{q-1}{2}}-1$ is non-zero if and only if $a$ is a non-square in $\F_q$. Hence, $w_{m,n}(\Sl_2(q))$ contains only semisimple elements with determinant of the form $-4^{\gamma-1}a^{\gamma}$, for some $a$ non-square. Since $\gamma$ is odd, the map $f_{\gamma} \colon \F_q^* \to \F_q^*$, restricts to $\bar{f_{\gamma}} \colon \F_q \backslash (\F_q)^2 \to \F_q \backslash (\F_q)^2$. Also, this map is a $\gamma$-one map and hence its image has size $\frac{q-1}{2\gamma}$. Therefore, the number of semisimple orbits in $w_{m,n}(\Sl_2(q))$ contains only $\frac{q-1}{2\gamma}$ semisimple orbits, all of them are anisotropic.
\end{proof}

\begin{corollary}\label{cor: 4l+3 one split and one anisotropic}
If $q \equiv 3 \mod{4}$, then there exist infinitely many Lie polynomials whose image over $\Sl_2(q)$ contains exactly two orbits, out of which one is a split semisimple orbit and another one is an anisotropic semisimple orbit (and no nilpotents).
\end{corollary}
\begin{proof}
If $q =3$, then $\Sl_2(q)$ contains exactly one split semisimple orbit and one anisotropic semisimple orbit. For each $i,j$ with $0<2<j<i$ and $i-j \equiv 2 \mod{4}$, we consider the words $f(A,X):= ad_A^i(X) - ad_A^j(X)$. Then, $f(A,X) = 2^{i-1}A^{i-1}ad_A(X) - 2^{j-1}A^{j-1}ad_A(X) = 2^{j-1}A^{j-1}(2^{i-j}A^{i-j}-1)ad_A(X) = 2^{j-1}A^{j-1}(A^{i-j}-1)ad_A(X)$. If $\det(A)=0$, then $A^2 =0$, thus $f(A,X) =0$. Also, for $A$ to be split semisimple and $q=3$, $A^2=I$. Thus, $f(A,X) =0$ in this case as well. So, we let $A = e+af$ with $a \notin {\F_q^*}^2$. Since $q=3$, we have $A =e-f$. Then $A^2 = -I$ and $f(A,X) = 2^{j-1}A^{j-1}((-1)^{\frac{i-j}{2}}-1)ad_A(X)$. As $i-j \equiv 2 \mod{4}$, $f(A,X) = -2^{j}A^{j-1}ad_A(X)$. For $X = b_1e+b_2f+b_3h \ne 0$ and $A = e-f$, $ad_A(X) = [e-f,b_1e+b_2f+b_3h] = -2b_3e-2b_3f+(b_1+b_2)h$ with determinant $-(b_1+b_2)^2 -4b_3^2$. Thus, $\det (f(A,X)) = -4^{j}(-1)^{j-1} ((b_1+b_2)^2 + 4b_3^2)$, which is zero if and only if $(b_1+b_2)^2 + 4b_3^2= 0$. This gives $-1$ is a square in $\F_3$, which is a contradiction. Hence, the image has no nilpotents. Further, the possible determinant values taken by $f(A,X)$ are $4^{j}(-1)^{j-1}b$, $b \in \F_3$. Hence, all determinant values are attained by the elements of the image. Consequently, the image contains exactly one split semisimple orbit, one anisotropic semisimple orbit, and no nilpotents.  

Now, let $q = 4\ell + 3$ for some $\ell \in \mathbb N$. Then $2\ell +1$ is an odd divisor of $q-1$. Hence, if $w_{m,n}$ satisfies the hypothesis of Proposition~\ref{Prop: few semisimple orbits for odd gamma} for c=2, then for $\gamma = 2\ell +1$, the set of non-zero determinant values attained by elements of $w_{m,n}(\Sl_2(q))$ is $\{-4^{2\ell}a^{2\ell +1} \mid a \in \F_q\} = \{4^{2\ell}, -4^{2\ell}\}$. Evidently, $\Sl_2(q)$ contains exactly one split semisimple orbit corresponding to the determinant value $-4^{2\ell}$, and exactly one anisotropic semisimple orbit corresponding to the determinant value $4^{2\ell}$, and no nilpotents.
\end{proof}

\begin{corollary}\label{Coro: exactly one anisotropic orbit}
Let $q \equiv 3 \mod{4}$. Let $\mathcal O$ be an anisotropic semisimple orbit of $\Sl_2(\F_q)$. Then there exists Lie polynomials whose image over $\Sl_2(\F_q)$ is $\mathcal{O} \cup \{0\}$.
\end{corollary}
\begin{proof}
Let $\mathcal O$ contain elements with determinant equal to $\lambda^2$, for some fixed $\lambda \in \F_q\backslash\F_q^2$. Let $w$ be any Lie polynomial satisfying the hypothesis of Proposition \ref{Prop: few semisimple orbits for odd gamma} for $c=q-1$. Then for $\gamma = \frac{q-1}{2}$, the image is $w\big(\Sl_2(\F_q)\big) = \{A \in \Sl_2(\F_q) \mid \det(A)= 4^{\frac{\gamma -3}{2}}\} \cup \{0\}$. Thus, the Lie polynomial $2^{-\frac{\gamma -3}{2}}\lambda w$ contains semisimple elements with determinant equal to $4^{-\frac{\gamma -3}{2}}\lambda^2 (4^{\frac{\gamma -3}{2}}) = \lambda^2$. Consequently, $2^{-\frac{\gamma -3}{2}}\lambda w\big(\Sl_2(\F_q)\big)$ is $\mathcal{O} \cup \{0\}$.  
\end{proof}

\begin{lemma}\label{prop: few semisimple orbits [w_1, w_m,n]}
Consider the Lie polynomials $w_{m,n}$ for $m\geq 1$, where (i) $i_r + j_r$ is odd for each $r$; (ii) $\alpha_m-\beta_m$ is positive and even for all $m$; (iii) $\beta_1 + \beta_2 +n$ is odd. Now we consider the Lie polynomial
$$w_{m,n}^0 = \bigg[\left[ad_A^{\alpha_0}(X) - ad_A^{\beta_0}(X), [ad_A^{i_0}(X), ad_A^{j_0}(X)] \right], w_{m,n} \bigg]$$
where $i_0 + j_0$, $\beta_0$ and $n$ are odd, $\alpha_0 - \beta_0$ is positive and even. Further assume $\sum_{s=0}^m \beta_s + \sigma_0 -2n - m > 4$ where $\sigma_0 = \sum_{r=0}^n (i_r+j_r)$. If $\beta_3, \beta_4, \dots, \beta_{m-1}$ and $m$ are odd, and $\beta_{m}$ is even, then the image of the word $w_{m,n}^0$ over $\Sl_2(q)$ contains only semisimple elements with determinant equal to 
$$-4^{\sum_{s=0}^m\beta_s+\sigma_0 -1} a^{\sum_{s=0}^m \beta_s + \sigma_0 -2n-m-3} \big(2^{\alpha_0-\beta_0}a^{\frac{\alpha_0-\beta_0}{2}}-1\big)^2\dots
\big(2^{\alpha_m-\beta_m}a^{\frac{\alpha_m-\beta_m}{2}}-1\big)^2 b^{2n + m +3},$$
where $a,b \in \F_q$.
\end{lemma}
\begin{proof}
Let us simply denote by $w = w_{m,n}^0$. We let $A = e+af$ and $X=b_1e+b_2f+b_3h$, for some $a,b_1,b_2,b_3 \in \F_q$. By Proposition~\ref{prop: w_n image} and Lemma~\ref{prop: few semisimple orbits w_m,n}, 
\begin{align*}
&w(A,X) = -2^{\beta_0+ i_0+j_0 + \sum_{s=1}^m \beta_s + \sigma -2} a^{\frac{\beta_0 + i_0+j_0 + \sum_{s=1}^m\beta_s + \sigma -2n-m-4}{2}} \big(2^{\alpha_0-\beta_0} a^{\frac{\alpha_0-\beta_0}{2}}-1\big)\\& \big(2^{\alpha_1-\beta_1}a^{\frac{\alpha_1-\beta_1}{2}}-1\big)\cdots \big(2^{\alpha_m-\beta_m}a^{\frac{\alpha_m-\beta_m}{2}}-1\big)\big((ab_1-b_2)^2-4ab_3^2 \big)^{n+1} [A\ ad_A(X), (ad_A(X))^m]\\
&= -2^{\sum_{s=0}^m\beta_s+\sum_{r=0}^n (i_r+j_r)-2}a^{\frac{\sum_{s=0}^m \beta_s + \sum_{r=0}^n (i_r+j_r)-2n-m-4}{2}} \big(2^{\alpha_0-\beta_0}a^{\frac{\alpha_0-\beta_0}{2}}-1\big) \cdots \\&  \cdots \big(2^{\alpha_m-\beta_m}a^{\frac{\alpha_m-\beta_m}{2}}-1\big)\big((ab_1-b_2)^2-4ab_3^2 \big)^{n+1}(2A (ad_A(X))^{m+1})\\
& = -2^{\sum_{s=0}^m\beta_s+\sum_{r=0}^n (i_r+j_r)-1}a^{\frac{\sum_{s=0}^m \beta_s + \sum_{r=0}^n (i_r+j_r)-2n-m-4}{2}} \big(2^{\alpha_0-\beta_0}a^{\frac{\alpha_0-\beta_0}{2}}-1\big) \cdots \\& \cdots \big(2^{\alpha_m-\beta_m}a^{\frac{\alpha_m-\beta_m}{2}}-1\big)\big((ab_1-b_2)^2-4ab_3^2 \big)^{n+1}A (ad_A(X))^{m+1}.
\end{align*}
Thus,
\begin{align*}
&\det(w(A,X)) =  4^{\sum_{s=0}^m\beta_s+\sum_{r=0}^n (i_r+j_r)-1} a^{\sum_{s=0}^m \beta_s + \sum_{r=0}^n (i_r+j_r)-2n-m-4}\big(2^{\alpha_0-\beta_0}a^{\frac{\alpha_0-\beta_0}{2}}-1\big)^2\cdots \\& \cdots \big(2^{\alpha_m-\beta_m}a^{\frac{\alpha_m-\beta_m}{2}}-1\big)^2\big((ab_1-b_2)^2-4ab_3^2 \big)^{2n+2}(-a) (-1)^{m+1}\big((ab_1-b_2)^2-4ab_3^2 \big)^{m+1}\\
& =  -4^{\sum_{s=0}^m\beta_s+\sum_{r=0}^n (i_r+j_r)-1} a^{\sum_{s=0}^m \beta_s + \sum_{r=0}^n (i_r+j_r)-2n-m-3} \big(2^{\alpha_0-\beta_0}a^{\frac{\alpha_0-\beta_0}{2}}-1\big)^2 \cdots \\& \cdots \big(2^{\alpha_m-\beta_m}a^{\frac{\alpha_m-\beta_m}{2}}-1\big)^2\big((ab_1-b_2)^2-4ab_3^2 \big)^{2n + m +3}.
\end{align*}
Thus,
\begin{align*}
\det(w(A,X)) = & -4^{\sum_{s=0}^m\beta_s+\sum_{r=0}^n (i_r+j_r)-1} a^{\sum_{s=0}^m \beta_s + \sum_{r=0}^n (i_r+j_r)-2n-m-3}\\ &\big(2^{\alpha_0-\beta_0}a^{\frac{\alpha_0-\beta_0}{2}}-1\big)^2\dots \big(2^{\alpha_m-\beta_m}a^{\frac{\alpha_m-\beta_m}{2}}-1\big)^2 b^{2n + m +3},
\end{align*}
for some $b \in \F_q$.
\end{proof}

\begin{proposition}\label{Prop: few semisimple orbits for even gamma}
Consider the Lie polynomial $w_{m,n}^0$ defined in previous Lemma~\ref{prop: few semisimple orbits [w_1, w_m,n]} with the assumption that $\alpha_s-\beta_s =2$ for all $s \in \{0,1,\dots,m\}$. Let $\gamma$ be any even divisor of $(q-1)$. If $\sum_{s=0}^m\beta_s + \sum_{r=0}^n(i_r+j_r) \equiv \gamma \mod (q-1)$, then for any $m$ and $n$ satisfying $2m +2 \equiv 0 \mod{(q-1)}$ and $2n + m + 3 \equiv 0 \mod{(q-1)}$, $w_{m,n}(\Sl_2(q))$ contains only semisimple elements with determinant equal to 
$$-4^{\gamma-1}a^{\gamma},$$
where $a \in \F_q$. Evidently, $w\big(\Sl_2(q)\big)$ contains only $\frac{q-1}{\gamma}$ semisimple orbits, all of them are split semisimple.
\end{proposition}

\begin{proof}
If $\alpha_s-\beta_s =2$ for all $s \in \{0,1,\dots,m\}$, then by Lemma~\ref{prop: few semisimple orbits [w_1, w_m,n]}, the Lie polynomial $w = [[ad_A^{\alpha_0}(X)-ad_A^{\beta_0}(X), [ad_A^{i_0}(X),ad_A^{j_0}(X)]], w_{m,n}]$ contains only semisimple orbits with determinant values equal to $-4^{\sum_{s=0}^m\beta_s+\sum_{r=0}^n (i_r+j_r)-1} a^{\sum_{s=0}^m \beta_s + \sum_{r=0}^n (i_r+j_r)-2n-m-3} \big(4a-1\big)^{2m+2} b^{2n + m +3}$, for some $a, b \in \F_q$. We first claim that $\sum_{s=0}^m\beta_s + \sum_{r=0}^n (i_r+j_r)$ is always even. This follows directly by the assumption that $i_r+j_r$ is odd for all $r$, $\beta_0, n, \beta_3, \beta_4, \beta_{m-1}$, $m$ are odd, $\beta_m$ and $\beta_1+\beta_2$ are even. Also, for any even integer $\gamma$, these variables can be chosen such that $\sum_{s=0}^m \beta_s + \sum_{r=0}^n (i_r+j_r) \equiv \gamma \mod (q-1)$. Now, if $\sum_{s=0}^m\beta_s + \sum_{r=0}^n(i_r+j_r) \equiv \gamma \mod (q-1)$, and $m$, $n$ satisfies $2m +2 \equiv 0 \mod (q-1)$ and $2n + m + 3 \equiv 0 \mod(q-1)$, then $w(\Sl_2(q))$ contains only semisimple elements with determinant equal to $-4^{\gamma-1}a^{\gamma}$,
where $a \in \F_q \backslash \{0, 4^{-1}\}$. Since $\gamma >1$, $\{-4^{\gamma-1}a^{\gamma} \mid a \in \F_q \backslash \{0, 4^{-1}\}\} = \{-4^{\gamma-1}a^{\gamma} \mid a \in \F_q^*\}$ which is of size $\frac{q-1}{\gamma}$. Finally, since $\gamma$ is even, $4^{\gamma-1}a^{\gamma}$ is always a square in $\F_q^*$. Consequently, $w\big(\Sl_2(q)\big)$ contains only split semisimple orbits. 
\end{proof}

\begin{corollary}
Let $\mathcal O$ be a split semisimple orbit of $\Sl_2(q)$. Then there exists Lie polynomials whose image over $\Sl_2(q)$ is $\mathcal{O} \cup \{0\}$.
\end{corollary}
\begin{proof} Let $\mathcal O$ contains elements with determinant equal to $-\lambda^2$, for some fixed $\lambda \in \F_q^*$. Let $w$ be any Lie polynomial satisfying the hypothesis of Proposition~\ref{Prop: few semisimple orbits for even gamma} for $\gamma = q-1$. Then $w\big(\Sl_2(q)\big) = \{A \in \Sl_2(q) \mid \det(A)= -4^{-1}\} \cup \{0\}$. Thus, the Lie polynomial $2\lambda w$ contains semisimple elements with determinant equal to $4\lambda^2 (-4^{-1}) = -\lambda^2$. Consequently, $2\lambda w\big(\Sl_2(q)\big)$ is $\mathcal{O} \cup \{0\}$.
\end{proof}

\section{Lie polynomial images on $\Sl_2(q)$}
In this section, we provide a partial answer to the Question~\ref{question2} for $\Sl_2(q)$ when $q$ is odd.  Kanel-Belov et. al. in~\cite{KMR17} showed that over an algebraically closed field of characteristic $\neq 2$, one can not get a single semisimple conjugacy class together with $0$ as an image. On the contrary, our Theorem~\ref{Thm: Image canditates are images} predicts that it is possible to get such classes as images over a finite field. Indeed, we produce Lie polynomials and get such classes as images. The results are summarised in the following Theorem:

\begin{theorem}\label{Thm: Semisimple orbits for each integer}
Consider the Lie algebra $\Sl_2(q)$, with $q$ odd. Let $t$ be a positive integer. Then, we have the following:
\begin{enumerate}
\item[(1)] The set of all nilpotent elements (and no semisimple elements) is an image of some Lie polynomial. 
\item[(2)] There exist infinitely many fields $\F_p$ of prime order such that for each such field, there exist infinitely many Lie polynomials whose image over $\Sl_2(p)$ contains no nilpotents, exactly $\frac{t}{2}$ split semisimple orbits, and $\frac{t}{2}$ anisotropic semisimple orbits.
\item[(3)] There exist infinitely many fields $\F_p$ of prime order such that for each such field, there exist infinitely many Lie polynomials whose image over $\Sl_2(p)$ contains no nilpotents, no anisotropic semisimple elements, and exactly $t$ split semisimple orbits.
\item[(4)] There exist infinitely many fields $\F_p$ of prime order such that for each such field, there exist infinitely many Lie polynomials whose image over $\Sl_2(p)$ contains no nilpotents, no split semisimple elements, and exactly $t$ anisotropic semisimple orbits.
\end{enumerate}
\end{theorem}
\begin{proof}
The proof of (1) is done in Proposition~\ref{prop:nilotent}. 

{\bf Proof of (2):} Let $t$ be an even integer. Then $t+1$ and $2t$ are coprime. Thus, by Dirichlet’s theorem on arithmetic progressions, there exist infinitely many primes of the form $t+1 + \ell(2t) = (2 \ell +1)t +1$, $\ell \in \mathbb N$. For each such prime $p$, $p-1$ has an odd divisor $2 \ell +1$ and $\frac{p-1}{2 \ell +1} =t$. Thus, if $w_{m,n}$ is a Lie polynomial satisfying the hypothesis of Proposition \ref{Prop: few semisimple orbits for odd gamma} for c =2 with $q = p = (2 \ell +1)t +1$ and $\gamma = 2 \ell +1>1$, then $w_{m,n}(\Sl_2(p))$ contains no nilpotents, exactly $\frac{t}{2}$ split semisimple orbits, and $\frac{t}{2}$ anisotropic semisimple orbits.

{\bf Proof of (3):} By Dirichlet’s theorem on arithmetic progressions, there exist infinitely many primes of the form $2\ell t+1$, $\ell \in \mathbb N$. For each such prime $p$, $p-1$ has an even divisor $2 \ell$ and $\frac{p-1}{2 \ell } =t$. Thus, if $w$ is a Lie polynomial satisfying the hypothesis of Proposition \ref{Prop: few semisimple orbits for even gamma} with $q = p = 2 \ell t +1$ and $\gamma = 2 \ell$, then $w(\Sl_2(p))$ contains no nilpotents, no anisotropic semisimple elements, and exactly $t$ split semisimple orbits.

{\bf Proof of (4):} Since $4t$ and $2t+1$ are coprime, by Dirichlet’s theorem on arithmetic progressions, there exist infinitely many primes of the form $4\ell t + 2t + 1$, $\ell \in \mathbb N$. For each such prime $p$, $p-1$ has an odd divisor $2 \ell +1$ and $\frac{p-1}{2(2 \ell +1)} = t$. Thus, if $w$ is a Lie polynomial satisfying the hypothesis of Proposition \ref{Prop: few semisimple orbits for odd gamma} for $c=q-1$, then with $q = p = 4\ell t + 2t + 1$ and $\gamma = 2 \ell +1$, $w(\Sl_2(\F_p))$ contains no nilpotents, no split semisimple elements, and exactly $t$ anisotropic semisimple orbits.
\end{proof}


\bibliographystyle{amsalpha}
\bibliography{Lie-algebra-maps}
\end{document}